\documentclass{article}
\pdfoutput=1
\usepackage{mystyle}
\usepackage{enumitem}
\usepackage{authblk}
\usepackage{amsmath,amsthm}
\usepackage{empheq}
\usepackage{graphicx}
\usepackage{bm}
\usepackage{subfigure}
\usepackage[algoruled,linesnumbered]{algorithm2e}

\SetCommentSty{mycommfont}  % algorithm comment font

\crefname{equation}{}{}

\newcommand{\customref}[2]{\textup{(}\hyperref[#1]{\textup{#2}}\textup{)}}

\title{Localization of point scatterers via sparse optimization on measures}
\date{\today}
\author{Giovanni S. Alberti, Romain Petit\footnote{Corresponding author (email: \texttt{romain.petit@edu.unige.it}).}~, Matteo Santacesaria}
\affil{MaLGa Center, Department of Mathematics, University of Genoa\\Via Dodecaneso 35, 16146 Genoa, Italy}

\begin{document}
\maketitle
\begin{abstract}
We consider the inverse scattering problem for time-harmonic acoustic waves in a medium with pointwise inhomogeneities. In the Foldy-Lax model, the estimation of the scatterers' locations and intensities from far field measurements can be recast as the recovery of a discrete measure from nonlinear observations. We propose a \enquote{linearize and locally optimize} approach to perform this reconstruction. We first solve a convex program in the space of measures (known as the Beurling LASSO), which involves a linearization of the forward operator (the far field pattern in the Born approximation). Then, we locally minimize a second functional involving the nonlinear forward map, using the output of the first step as initialization. We provide guarantees that the output of the first step is close to the sought-after measure when the scatterers have small intensities and are sufficiently separated. We also provide numerical evidence that the second step still allows for accurate recovery in settings that are more involved.
\end{abstract}

\section{Introduction}
The localization of inhomogeneities included in a medium has numerous applications in imaging. These inclusions might model defects in material science or tumors in medical imaging. A popular method for detecting them is to exploit scattering phenomena, which originate from the interaction of waves propagating in the medium with the inhomogeneities. By sending incident waves and measuring the scattered ones away from the medium, an observer can aim at estimating the localization and characteristics of the inclusions. This reconstruction problem is called the \emph{inverse scattering problem} \cite{coltonInverseAcousticElectromagnetic2012}. In this work, we focus on the case of time-harmonic acoustic incident waves.

In order to solve the inverse scattering problem with few measurements, one needs to make strong assumptions on the type of inhomogeneities included in the medium. The simplest case is arguably the one of very small inclusions, which can be approximately modeled by point-like inhomogeneities. This particular setting has been the subject of numerous works; see for example \cite{devriesPointScatterersClassical1998,martinMultipleScatteringInteraction2006,albeverioSolvableModelsQuantum2012}. The classical model for the scattering of acoustic waves by point-like inclusions, known as \emph{Foldy-Lax model}, was formally introduced in~\cite{foldyMultipleScatteringWaves1945} (see also \cite{laxMultipleScatteringWaves1951}). In this framework, the inverse scattering problem can be recast as the recovery of a discrete measure, which encodes the locations and the intensities of the scatterers, from nonlinear measurements.

In this article, we propose to connect this inverse problem to the \emph{sparse spikes problem}, which has attracted a lot of attention in the past ten years (see e.g. \cite{decastroExactReconstructionUsing2012,brediesInverseProblemsSpaces2013} and the review \cite{lavilleOffTheGridVariationalSparse2021}). These works investigated the possibility of recovering a discrete measure from noisy linear observations. Although the measurements we consider in the context of inverse scattering are nonlinear, we propose to investigate how the guarantees and reconstruction methods developed for the sparse spikes problem can still be leveraged.

\subsection{Problem formulation}
Before introducing the Foldy-Lax model for the scattering of acoustic waves by point-like inhomogeneities \cite{albeverioSolvableModelsQuantum2012}, we briefly describe the case of a piecewise continuous perturbation of the refractive index $q$. Let $d\in\{2,3\}$ be the ambient dimension. Considering time-harmonic acoustic incident waves with wavenumber~$\kappa>0$, the incident field $u^{\mathrm{in}}$ associated to an incident direction $\theta\in\mathbb{S}^{d-1}$ is given by
\begin{equation}
    \label{eq:uin}
    u^{\mathrm{in}}(y)=e^{i\kappa\theta\cdot y}\,.
\end{equation}
The total field $u^{\mathrm{tot}}$ measured by the observer, which is the sum of the incident field $u^{\mathrm{in}}$ and the scattered field~$u^s$, solves the Helmoltz equation \eqref{pde} with Sommerfeld radiation condition \eqref{radiation_cond}
\begin{align}[left=\empheqlbrace]
    &\Delta u^{\mathrm{tot}}(y)+\kappa^2(1+q(y))u^{\mathrm{tot}}(y)=0~~\mathrm{in}~\RR^d,\label{pde}\\
    &\underset{|y|\to+\infty}{\mathrm{lim}}~|y|^{(d-1)/2}\left(\frac{\partial u^s}{\partial |y|}(y)-i\kappa u^s(y)\right)=0\,.\label{radiation_cond}
\end{align}
Now, let us denote the intensities and the locations of a set of $s$ point scatterers by $a\in(\mathbb{C}^*)^s$ and $x\in(\RR^d)^s$, respectively, where $\mathbb{C}^*\eqdef \mathbb{C}\setminus\{0\}$ denotes the set of nonzero complex numbers. The Foldy-Lax model formally corresponds to taking $q=\sum_{i=1}^s a_i\delta_{x_i}$ in \eqref{pde}. In this model, the total field~$u^{\mathrm{tot}}$ is given by
\begin{equation*}
	u^{\mathrm{tot}}(y)=u^{\mathrm{in}}(y)+\sum\limits_{i=1}^s G(y,x_i) a_i u_i,\quad y\in\RR^d,
\end{equation*}
where $G$ is the Green's function of the Helmoltz equation (see \Cref{sec_notations}) and the coefficients~${u=(u_i)_{1\leq i \leq s}}$ are the solutions of the Foldy-Lax system
\begin{equation}
	u_i=u^{\mathrm{in}}(x_i)+\kappa^2\sum\limits_{\substack{j=1\\j\neq i}}^s G(x_i,x_j)a_j u_j,\quad i\in\{1,...,s\}.
	\label{foldy}
\end{equation}
We assume that $u^{\mathrm{tot}}$ can be measured by the observer at infinity, i.e.\ that we have access to the far field pattern~$u^{\infty}(\hat{x},\theta)$ for several incident directions $\theta$ and observations directions $\hat{x}$, given by
\begin{equation}
	u^{\infty}(\hat{x},\theta)\eqdef \frac{\kappa^2}{4\pi}\sum\limits_{i=1}^s a_{i} u_i e^{-i\kappa \hat{x}\cdot x_i},\quad (\hat{x},\theta)\in\mathbb{S}^{d-1}\times\mathbb{S}^{d-1}.
	\label{far_field}
\end{equation}

Our goal is to recover $a$ and $x$ from the knowledge of a possibly noisy version of $(u^{\infty}(\hat{x},\theta))_{(\hat{x},\theta)\in\Omega}$, with~$\Omega$ a finite subset of $\mathbb{S}^{d-1}\times\mathbb{S}^{d-1}$. In this work, we consider the case where $\Omega=\{(\hat x_k,\theta_k)\}_{k=1}^m$, for a suitable choice of incident directions $\theta_k$ and corresponding observation directions $\hat x_k$. In other words, we consider $m$ incident waves, and for each incident wave we measure the corresponding far-field pattern at a single point at infinity. We stress that the difficulty of this task stems from the nonlinear dependence of $u^{\infty}$ on $a$ and $x$ and on the availability of only finitely many measurements.

\paragraph{Derivation of the Foldy-Lax model as a vanishing size limit.}  As a rule of thumb, the point approximation for small inhomogeneities is valid as soon as they have sub-wavelength size. The rigorous derivation of the Foldy-Lax model as a limit case of scattering by inhomogeneities of vanishing size has been investigated in several works. If the homogeneities have a moderate contrast\footnote{See also \cite{capdeboscqExtendingRepresentationFormulas2022}, in which the same behavior arises from a particular scaling between the (high) contrast and the size of the inhomogeneities.}, the first term in the asymptotic expansion of the scattered field is indeed the one given by the Foldy-Lax model, but it goes to zero in the vanishing size limit~(see e.g.~\cite{cassierMultipleScatteringAcoustic2013,challaJustificationFoldyLax2014,bendaliApproximationMultipolesMultiple2016}). On the other hand, there exist critical scalings between the size and the contrast of the inhomogeneities under which \emph{resonating frequencies} appear. When using incident frequencies close to those, the Foldy-Lax model is valid and the limit scattered field is non-trivial (see e.g. \cite{ammariMinnaertResonancesAcoustic2018,ammariPointinteractionApproximationFields2019,ammariEquivalentMediaGenerated2020,alsenafiFoldyLaxApproximation2022,mantileOriginMinnaertResonances2022,ammariMathematicalAnalysisElectromagnetic2023}). In the acoustic setting, such resonance phenomena (called \emph{Minnaert resonances}) have important applications in medical ultrasonic imaging~\cite{erricoUltrafastUltrasoundLocalization2015}, where the small scatterers model micro bubbles. We also mention their use in optics, where the scatterers model dielectric nanoparticles~\cite{ammariSubwavelengthResonantDielectric2019}.

\subsection{Related works}

\paragraph{Reconstruction methods in inverse acoustic scattering.} There is a wide literature on the numerical resolution of the inverse acoustic scattering problem. We  mention here the three main approaches, namely iterative methods, decomposition methods, and sampling methods (see for example \cite[Sections 4-6 of Chapter~5 and~Sections 3-5 of Chapter 10]{coltonInverseAcousticElectromagnetic2012}). Iterative methods (see for example \cite{haberOptimizationTechniquesSolving2000,bakushinskyIterativeMethodsApproximate2005,kaltenbacherIterativeRegularizationMethods2008,kaltenbacherIterativeMethodsNonlinear2009}) solve the equation associated to the inverse problem~(in the case of noiseless observations) or a relevant optimization problem (in the case of noisy observations) by iteratively updating an initial guess of the inhomogeneity. Their main drawback is that, due to the nonlinearity of the forward map, convergence guarantees are only local. As a consequence, in order to provide a good initial guess, one needs extensive a priori information on the distribution of the scatterers. Decomposition methods, introduced in \cite{kirschIntegralEquationFirst1986,kirschNumericalMethodInverse1987,kirschOptimizationMethodInverse1987}, split the inverse problem into an ill-posed linear problem and a well-posed nonlinear problem. Just as iterative methods, they only offer local recovery guarantees, but avoid the high cost of solving the full forward problem at each iteration. Sampling methods~(see e.g.~\cite{coltonSimpleMethodSolving1996,cakoniQualitativeMethodsInverse2005,kirschFactorizationMethodInverse2007}) do not require extensive a priori knowledge of the inhomogeneity, but are of qualitative type: they mainly provide information on the location of the scatterers, but not on the perturbation of the refraction index. Finally, let us mention \cite{bellisAcousticInverseScattering2013}, which introduces a qualitative reconstruction method based on the notion of topological derivative.

\paragraph{Reconstruction of small inhomogeneities.} The reconstruction of small inhomogeneities has also attracted a lot of attention outside the Foldy-Lax model and the specific case of acoustic waves. We refer the reader to \cite{ammariReconstructionSmallInhomogeneities2004} and to the long list of references therein concerning the localization of small conductivity, elastic and electromagnetic inclusions. In these works, the contrast of the inhomogeneities is moderate, and the measurements are hence close to those corresponding to the medium without the inclusions. Reconstruction methods are mainly based on asymptotic expansions of the perturbation resulting from the presence of the inhomogeneities.

\paragraph{Born approximation.} When the scatterers have small intensities and are sufficiently separated, the solution of the Foldy-Lax system is close to $(u^{\mathrm{in}}(x_i))_{1\leq i\leq s}$ (see \Cref{sec_bound_lin_err} for more details). As a result, the far field pattern is close to the Fourier transform of~${\mu=\sum_{i=1}^s a_i \delta_{x_i}}$. Consequently, the recovery of $a$ and $x$ is a particular instance of the sparse spikes problem mentioned above. In the context of inverse scattering, most reconstruction methods are based on the MUltiple SIgnal Classification~(MUSIC) algorithm. This algorithm, first used in signal processing, was introduced for inverse scattering in \cite{lev-ariTimereversalTechniqueReinterpreted2000} (see also \cite[Chapter 10]{devaneyMathematicalFoundationsImaging2012}). As pointed out in e.g. \cite{cheneyLinearSamplingMethod2001,kirschFactorizationMethodInverse2007}, it has strong connections with sampling methods. Its noise robustness has been extensively studied in the literature on the sparse spikes problem. However, in the context of inverse scattering, we are not aware of any work leveraging these results to provide quantitative reconstruction guarantees accounting for the linearization error.

\paragraph{Inverse scattering in the nonlinear case.} To our knowledge, in the Foldy-Lax setting, most works also rely on the MUSIC algorithm \cite{gruberTimereversalImagingMultiple2004,devaneyTimereversalbasedImagingInverse2005,challaInverseScatteringPointlike2012}. Exact recovery guarantees exist in the noiseless case, but we are not aware of any sensitivity analysis in the case of noisy measurements. In \cite{gilbertNonlinearIterativeHard2020}, the scatterers are constrained to belong to a fixed finite grid on the domain. The convergence of a nonlinear iterative hard thresholding algorithm is then investigated. The authors obtain global convergence guarantees, under the assumption that the iterates remain bounded and some coherence estimates are satisfied by the forward operator. Checking the validity of these assumptions is however involved, and is the subject of an intensive numerical investigation.

\subsection{Contributions}
We propose a \enquote{linearize and locally optimize} approach to solve the considered inverse problem. First, we solve a convex program on the space of measures, which involves the linearized forward operator (the far field pattern in the Born approximation). Then, we use this first estimate to initialize a local optimization procedure involving the nonlinear forward operator.

This approach allows us to leverage several results on the recovery of a discrete measure from noisy linear measurements (or sparse spikes problem). To our knowledge, their application to inverse scattering is new. In particular, they allow us to obtain quantitative recovery guarantees, as well as lower bounds on the number of required measurements. To this aim, we show that the assumptions of \cite[Theorem~3]{poonGeometryOfftheGridCompressed2023} and \cite[Theorem 1]{poonSupportLocalizationFisher2019} are satisfied in the case of continuous Fourier sampling with frequencies uniformly distributed in a Euclidean ball of fixed radius. This setting was not covered by previous works, and is a result of independent interest.

We also extensively study the error committed by linearizing the forward operator. We treat the special case of two scatterers, and derive two different bounds in the general case. Numerical evidence suggests that their dependence on the wave number and on the minimal separation between the spikes is sharp.

We developed a JAX-based Python package implementing the proposed reconstruction method, which is publicly available at~\url{https://github.com/rpetit/pointscat}. It can run seamlessly on CPUs, GPUs and TPUs, and benefits from just-in-time compilation. 

\paragraph{Structure of the paper.} In \Cref{sec_sparse_spikes}, we provide a brief introduction to the sparse spikes problem. Then, we give an overview of the approach we propose to solve the inverse problem in \Cref{sec:overview}. \Cref{sec:rec} is dedicated to the derivation of theoretical guarantees for the linear part of our recovery procedure. Finally, we introduce our reconstruction algorithm in \Cref{sec_algo} and present numerical results on a few examples in \Cref{sec_num_res}.

\subsection{Notations}
\label{sec_notations}
\paragraph{General notations.} Let $\mathcal{X}$ be a connected bounded open subset of $\RR^d$ and denote the space of complex-valued Radon measures on $\mathcal{X}$  by $\mathcal{M}(\mathcal{X})$. The total variation of a measure $\nu\in\mathcal{M}(\mathcal{X})$ is denoted by $|\nu|(\mathcal{X})$. If $a\in\mathbb{C}^s$ and $1\leq i \leq s$, we define 
\begin{equation*}
	a_{-i}\eqdef(a_1,...,a_{i-1},a_{i+1},..,a_s)\in\mathbb{C}^{s-1}\,.
\end{equation*} 
When $x=(x_1,...,x_s)\in\mathcal{X}^s$, we sometimes use the notation $u^{\mathrm{in}}(x)\eqdef (u^{\mathrm{in}}(x_i))_{1\leq i\leq s}$. Finally, we use the notation $|\cdot|$ for the Euclidean norm in $\RR^d$.

\paragraph{Forward operators.} Given $s$ scatterers with intensities $a\in(\mathbb{C}^*)^s$ and locations $x\in\mathcal{X}^s$, we denote  the far field pattern in the Born approximation by $u^{\infty,b}$. Recalling \eqref{eq:uin} and \eqref{far_field}, the Born approximation~$u\approx u^{\mathrm{in}}(x)$ yields
\begin{equation}
	u^{\infty,b}(\hat{x},\theta)\eqdef\frac{\kappa^2}{4\pi}\sum\limits_{i=1}^s a_i e^{-i\kappa(\hat{x}-\theta)\cdot x_i},\quad (\hat{x},\theta)\in\mathbb{S}^{d-1}\times\mathbb{S}^{d-1}.
	\label{far_field_born}
\end{equation}
Given $m$ incident directions $(\theta_k)_{1\leq k\leq m}$ and corresponding  observation directions $(\hat{x}_k)_{1\leq k\leq m}$, we denote the associated forward operator on $\mathcal{M}(\mathcal{X})$ by $\Phi^b$, which is
\begin{equation*}
	\begin{aligned}
		\Phi^b \colon \mathcal{M}(\mathcal{X}) &  \to \mathbb{C}^m \\
		\mu & \mapsto \frac{1}{\sqrt{m}}\left[\int_{\mathcal{X}}e^{-i\kappa(\hat{x}_k-\theta_k)\cdot y}\,d\mu(y)\right]_{k=1}^m.
	\end{aligned}
\end{equation*}
Note that $\Phi^b$ is a linear map. Finally, we define the Born and Foldy forward operators associated to the arrangement of scatterers $x$:
\begin{equation*}
	\begin{aligned}
		\Phi^b_x \colon \mathbb{C}^s &  \to \mathbb{C}^m \\
		a & \mapsto \frac{1}{\sqrt{m}}\left[u^{\infty,b}(\hat{x}_k,\theta_k)\right]_{k=1}^m,
	\end{aligned}
	~~~~~
	\begin{aligned}
		\Phi^f_x \colon \mathbb{C}^s &  \to \mathbb{C}^m \\
		a & \mapsto \frac{1}{\sqrt{m}}\left[u^{\infty}(\hat{x}_k,\theta_k)\right]_{k=1}^m.
	\end{aligned}
\end{equation*}
With this notation, taking $\mu=\sum_{i=1}^s a_i\delta_{x_i}$ yields $\Phi^b_x a=\Phi^b\mu$.

\paragraph{Green functions.} We denote the Green function associated to the Helmoltz equation with wavenumber~$\kappa$ by  $G$, i.e.
\begin{equation*}
	G(x,y)\eqdef\begin{cases}
		\frac{i}{4}H_0^{(1)}(\kappa|x-y|)&\mathrm{if}~d=2\,,\\\frac{1}{4\pi|x-y|}e^{i\kappa|x-y|}&\mathrm{if}~d=3\,,
	\end{cases}\quad x,y\in\RR^d,
\end{equation*}
where $H_0^{(1)}$ denotes the Hankel function of the first kind and order zero. As we often need to upper bound the modulus of $G$, by \cite[page 446]{watsonTreatiseTheoryBessel1995} and \cite[Theorem A]{freitasSharpBoundsModulus2018} we have 
\begin{equation}
\label{eq:Gphi}
\forall x,y\in\RR^d,~|G(x,y)|\leq \phi(|x-y|),
\end{equation}
where $\phi\colon\RR_+^*\to \RR$ is the non-increasing function given by
\begin{equation}
	\phi(t)\eqdef\begin{cases}
		\frac{1}{4}\mathrm{min}\left(\sqrt{1+\frac{4}{\pi^2}(\gamma+\mathrm{log}(\kappa t/2))^2},\sqrt{\frac{2}{\pi \kappa t}}\right)&\mathrm{if}~d=2,\\\frac{1}{4\pi t}&\mathrm{if}~d=3,
	\end{cases}\quad t\in\RR_+^*,
	\label{eq:phi}
\end{equation}
where $\gamma$ is Euler's constant. We stress that the function $\phi$ depends on $\kappa$ when $d=2$.

\paragraph{Foldy matrix.} Given $a\in\mathbb{C}^s$ and $x\in(\mathbb{R}^{d})^s$, we define the matrix $T_{a,x}\in\mathbb{C}^{s\times s}$ by
\begin{equation*}
	(T_{a,x})_{ij}=\begin{cases}0&\mathrm{if}~i=j,\\G(x_i,x_j)a_j&\mathrm{otherwise,}\end{cases}\quad (i,j)\in\{1,...,s\}^2.
\end{equation*}
With this definition, the Foldy system \eqref{foldy} writes
\begin{equation}
\label{eq:foldycompact}
(Id-\kappa^2T_{a,x})u=u^{\mathrm{in}}(x)\,.
\end{equation}

\section{Background on the sparse spikes problem}
\label{sec_sparse_spikes}
Since our approach heavily relies on existing results on the sparse spikes problem,  in this section we provide a brief introduction to this topic. We also refer the reader to \cite{denoyelleTheoreticalNumericalAnalysis2018,catalaRelaxationsSemidefiniesPositives2020,lavilleOffTheGridVariationalSparse2021} and to the lecture notes \cite{poonIntroductionSparseSpikes2019}.

\paragraph{Problem formulation.} The sparse spikes problem consists in recovering a discrete measure~${\mu=\sum_i a_i\delta_{x_i}}$ from the knowledge of $y=\Phi\mu+w$, where $\Phi\colon\mathcal{M}(\mathcal{X})\to\CC^m$ is linear and $w\in\CC^m$ is an additive noise. This task naturally appears in a wide range of applications. Let us mention the localization of point light sources in biological and astronomical imaging \cite{denoyelleSlidingFrankWolfe2019}, the estimation of the parameters of mixture models in statistics~\cite{kerivenSketchingLargescaleLearning2018,castroSuperMixSparseRegularization2021}, and the training of two-layer neural networks in machine learning \cite{bachBreakingCurseDimensionality2017}. The most popular approaches for solving it are either of variational or of parametric type. In this work, we focus on the former. The latter is however briefly discussed below.

\paragraph{Parametric approaches.} Broadly speaking, parametric approaches consist in encoding the locations of the atoms in the roots of some polynomial. The most popular are Prony's method, Estimation of Signal Parameters via Rotational Invariance Techniques (ESPRIT), and the MUSIC algorithm. We refer the reader to \cite[Section 0.2]{catalaRelaxationsSemidefiniesPositives2020} for a detailed presentation and a comparison with variational approaches.

\paragraph{The Beurling LASSO.} In this work, we focus on a variational approach. It consists in looking for an estimate of $\mu$ among the solutions of a convex optimization problem in the space of measures, called the \emph{Beurling LASSO} or \emph{BLASSO}. It was introduced in \cite{decastroExactReconstructionUsing2012,brediesInverseProblemsSpaces2013}, and writes
\begin{equation}
	\underset{\nu\in\mathcal{M}(\mathcal{X})}{\mathrm{min}}~\frac{1}{2}\|\Phi\nu-y\|_2^2+\lambda |\nu|(\mathcal{X})\,,
	\label{blasso}
\end{equation}
where $\lambda$ is a regularizer parameter that should be chosen according to an estimate of the noise level $\|w\|_2$. This problem can be seen as an infinite-dimensional version of the celebrated LASSO, which is a popular method for the recovery of sparse finite-dimensional signals.

\paragraph{Stable recovery.} Under suitable assumptions, the solutions of \cref{blasso} can be guaranteed to be close to~$\mu$. In~\cite{brediesInverseProblemsSpaces2013,candesSuperResolutionNoisyData2013,fernandez-grandaSupportDetectionSuperresolution2013,azaisSpikeDetectionInaccurate2015}, some convergence guarantees were proved in the low noise regime. In~\cite{duvalExactSupportRecovery2015,denoyelleSupportRecoverySparse2017,poonSupportLocalizationFisher2019}, stronger exact support recovery results were derived. Outside the low noise regime, guarantees were also obtained in~\cite{tangMinimaxLineSpectral2015,boyerAdaptingUnknownNoise2017,butuceaOffthegridLearningSparse2022} under a Gaussian noise assumption.

\paragraph{Numerical methods.} The algorithm we focus on for solving \cref{blasso}, introduced in \cite{denoyelleSlidingFrankWolfe2019} (see also the earlier works \cite{brediesInverseProblemsSpaces2013,boydAlternatingDescentConditional2017}), is based on the Frank-Wolfe or conditional gradient algorithm. It is known as the \emph{sliding Frank-Wolfe algorithm}. It consists in iteratively adding atoms to the support of the zero measure based on the residual observations, and comes with strong convergence guarantees. Let us mention that there exist other approaches for solving \cref{blasso}. These include conic particle gradient descent \cite{chizatSparseOptimizationMeasures2022}, projected gradient descent~\cite{benardEstimationOffthegridSparse2023}, proximal methods \cite{valkonenProximalMethodsPoint2023}, exchange algorithms \cite{flinthLinearConvergenceRates2021}, as well as semidefinite relaxations \cite{tangCompressedSensingGrid2013,candesMathematicalTheorySuperresolution2014,decastroExactSolutionsSuper2017,catalaLowRankApproachOfftheGrid2019}.

\section{Overview of the proposed approach}\label{sec:overview}
The inverse problem we wish to solve aims to recover $a\in(\mathbb{C}^*)^s$ and $x\in\mathcal{X}^s$ from $y=\Phi^f_{x} a+w^{\rm noise}$, where~$w^{\rm noise}\in\mathbb{C}^m$ is an additive noise. In this section, we describe our two-step ``linearize and locally optimize'' approach to this problem. We first focus on the first step, hereby called the \emph{linear step}, and then turn to the second step, hereby called the \emph{nonlinear step}.

\paragraph{Linear step.} The key idea of our approach is to decompose $\Phi^f_{x} a$ by writing $\Phi^f_{x} a=\Phi^b_{x} a+w^{\rm lin}$, where~$w^{\rm lin} $ is the \emph{linearization error}. As a result, we can recast the recovery of $a$ and $x$ from $y$ as the recovery of the discrete measure $\mu=\sum_{i=1}^s a_{i}\delta_{x_{i}}$ from the noisy linear measurements $\Phi^b\mu+w$, where~$w=w^{\rm lin}+w^{\rm noise}$:
\begin{equation*}
y=\Phi^f_{x} a+w^{\rm noise} =\Phi^b_{x} a+w^{\rm lin} +w^{\rm noise} =\Phi^b\mu+w^{\rm lin} +w^{\rm noise}\,.
\end{equation*}
This is a particular instance of the sparse spikes problem. We hence propose to estimate $\mu$ by solving the BLASSO program \cref{blasso} with linear forward map $\Phi=\Phi^b$:
\begin{equation}
	\underset{\nu\in\mathcal{M}(\mathcal{X})}{\mathrm{min}}~\frac{1}{2}\|\Phi^b\nu-y\|_2^2+\lambda^b |\nu|(\mathcal{X})
	\label{blasso2}
\end{equation}
for some regularization parameter $\lambda^b>0$.

\paragraph{Nonlinear step.} Approximating $\mu$ by solving $\cref{blasso2}$ yields reconstructions whose quality deteriorates as the norm of the linearization error $\|w^{\rm lin}\|_2$ grows. We hence propose to improve the estimate obtained in the linear step by performing a local minimization step involving the nonlinear forward operator $\Phi^f$. More precisely, we wish to locally optimize the non-convex functional $J^f$ defined by
\begin{equation}
\begin{aligned}
	J^f \colon \mathbb{C}^s\times \mathcal{X}^{s} &  \to \RR \\
	(\tilde{a},\tilde{x}) & \mapsto \frac{1}{2}\|\Phi^f_{\tilde{x}} \tilde{a}-y\|_2^2+\lambda^f\|\tilde{a}\|_1\,,
\end{aligned}
\label{blasso_obj_nonlin}
\end{equation}
by using the output of the linear step as initialization. Provided the linearization error is sufficiently small, we can hope that the linear estimate lies in a basin of attraction of the true intensities and locations.

\section{Recovery guarantees}\label{sec:rec}
In this section, we first study the error $w^{\rm lin}=\Phi^f_{x} a-\Phi^b_{x} a$ committed by linearizing the forward operator, i.e.\ the difference between the far field patterns associated to the Foldy model and to the Born approximation. We show that the error can be controlled provided the scatterers are sufficiently separated and their intensities are small. Then, we leverage existing results on the sparse spikes problem, and show that, in this regime, solving~\cref{blasso} allows us to stably recover $\mu$.

\subsection{Bounding the linearization error}
\label{sec_bound_lin_err}
In this subsection, we study the linearization error $\|u^{\infty}-u^{\infty,b}\|_{\infty}$, see \eqref{far_field} and \eqref{far_field_born}. Classical error bounds are derived in e.g.\ \cite{nattererErrorBoundBorn2004}, which, however, does not treat the case of point-like scatterers. In short, we show that the linearization error is small as soon as~$\kappa^2\phi(\Delta)\|a\|_1$ is small, where $\phi$ is given in \eqref{eq:phi} and 
\begin{equation}\label{eq:Deltadef}
\Delta = \min_{i\neq j} |x_i-x_j|\,,
\end{equation}
i.e.\ the scatterers are sufficiently separated and have small intensities.

We first investigate the special case of two scatterers.
\begin{proposition}
    Let $a\in(\mathbb{C}^*)^2$ and $x\in\mathcal{X}^2$. Define $\alpha=\kappa^2 |G(x_1,x_2)|\sqrt{|a_1||a_2|}$. If $\alpha<1$, then we have
    \begin{equation}
	\|u^{\infty}-u^{\infty,b}\|_{\infty}\leq \frac{\kappa^2}{4\pi} \frac{2\alpha}{1-\alpha^2}\Big(\alpha\frac{|a_1|+|a_2|}{2}+\sqrt{|a_1| |a_2|}\Big)\,.
	\label{two_bound}
\end{equation}
\label{prop_two_scat}
\end{proposition}
\begin{proof}
   For $s=2$ scatterers, the Foldy matrix $Id-\kappa^2 T_{a,x}$ can be explicitly inverted. Indeed, we have that~$\mathrm{det}(Id-\kappa^2T_{a,x})=1-\beta^2$ with 
\begin{equation*}
	\beta\eqdef \kappa^2G(x_1,x_2)\sqrt{|a_1| |a_2|}e^{i(\varphi_1+\varphi_2)/2}~\mathrm{and}~a_i = |a_i|e^{i\varphi_i}\,.
\end{equation*}
Therefore, provided $\beta^2\neq 1$, we obtain that $Id-\kappa^2 T_{a,x}$ is invertible and
\begin{equation*}
	(Id-\kappa^2 T_{a,x})^{-1}=\frac{1}{1-\beta^2}\begin{pmatrix}1&\kappa^2G(x_1,x_2)a_2\\\kappa^2G(x_1,x_2)a_1&1\end{pmatrix}.
\end{equation*}
In view of \eqref{eq:uin}, \eqref{far_field}  and \eqref{eq:foldycompact}, this yields
\begin{equation*}
	u^{\infty}(\hat{x},\theta)=\frac{\kappa^2}{4\pi}\left(a_1 e^{-i\kappa \hat{x}\cdot x_1}\frac{e^{i\kappa\theta\cdot x_1}+\kappa^2G(x_1,x_2)a_2 e^{i\kappa\theta\cdot x_2}}{1-\beta^2}+a_2 e^{-i\kappa \hat{x}\cdot x_2}\frac{e^{i\kappa\theta\cdot x_2}+\kappa^2G(x_1,x_2)a_1 e^{i\kappa\theta\cdot x_1}}{1-\beta^2}\right),
\end{equation*}
which, by \eqref{far_field_born}, finally allows us to get the following bound:
\begin{equation}
	\|u^{\infty}-u^{\infty,b}\|_{\infty}\leq \frac{\kappa^2}{4\pi} \frac{2|\beta|}{|1-\beta^2|}\Big(|\beta|\frac{|a_1|+|a_2|}{2}+\sqrt{|a_1| |a_2|}\Big)\,.
	\label{two_bound_prim}
\end{equation}
When $\alpha=|\beta|<1$, we finally obtain \eqref{two_bound}. 
\end{proof}
Next, we turn to the general case, in which we provide two different bounds. The second one is finer and explicitly depends on the strength of the pairwise interactions between the scatterers.
\begin{proposition} Let $a\in(\mathbb{C}^*)^s$ and $x\in\mathcal{X}^s$. Define $\Delta\eqdef \mathrm{min}_{i\neq j}\,|x_i-x_j|$ and~${\alpha\eqdef \kappa^2\phi(\Delta)\,\mathrm{max}_{i}\,\|a_{-i}\|_1}$. If~$\alpha<1$, then we have
\begin{equation}
	\|u^{\infty}-u^{\infty,b}\|_{\infty}\leq \frac{\kappa^2}{4\pi} \|a\|_1\frac{\alpha}{1-\alpha}
	\label{gen_bound_1}
\end{equation}
and
\begin{equation}
	\|u^{\infty}-u^{\infty,b}\|_{\infty}\leq \frac{\kappa^2}{4\pi}\bigg(\|a\|_1\frac{\alpha^2}{1-\alpha}+2\sum_{i<j} |a_i| |a_j| \kappa^2\phi(|x_i-x_j|)\bigg).
	\label{gen_bound_2}
\end{equation}
\label{prop_gen_scat}
\end{proposition}
\begin{proof}
The infinity operator norm of $T_{a,x}$, defined by
\begin{equation*}
	\|T_{a,x}\|_{\infty}\eqdef\mathrm{sup}~\{\|T_{a,x}u\|_{\infty}:\|u\|_{\infty}=1\}\,,
\end{equation*}
can be bounded as follows (see \eqref{eq:Gphi} and \eqref{eq:Deltadef}):
\begin{equation*}
\|T_{a,x}\|_{\infty}=\underset{1\leq i\leq s}{\mathrm{max}}~\sum\limits_{\substack{j=1\\j\neq i}}^s |G(x_i,x_j)||a_j|\leq \phi(\Delta)\,\underset{1\leq i\leq s}{\mathrm{max}}~\|a_{-i}\|_1\leq \phi(\Delta)\|a\|_1\,.
\end{equation*}
Provided $\kappa^2\|T_{a,x}\|_{\infty}<1$, by \eqref{eq:uin}, \eqref{far_field}, \eqref{far_field_born},  and \eqref{eq:foldycompact} this yields
\begin{equation*}
	\begin{aligned}
		|u^{\infty}(\hat{x},\theta)-u^{\infty,b}(\hat{x},\theta)|&=\frac{\kappa^2}{4\pi}\left|\sum\limits_{i=1}^s a_i e^{-i\kappa \hat{x}\cdot x_i}\left[((Id-\kappa^2 T_{a,x})^{-1}-Id)u^{\mathrm{in}}(x)\right]_i\right|
\\ &\leq \frac{\kappa^2}{4\pi}\|a\|_1\|(Id-\kappa^2 T_{a,x})^{-1}-Id\|_{\infty}\\&\leq \frac{\kappa^2}{4\pi}\|a\|_1 \Bigg\|\sum_{n=1}^{+\infty}\kappa^{2n} T_{a,x}^n \Bigg\|_{\infty}\\ &\leq \frac{\kappa^2}{4\pi}\|a\|_1 \sum\limits_{n=1}^{+\infty}(\kappa^2 \|T_{a,x}\|_{\infty})^n\\&=\frac{\kappa^2}{4\pi}\|a\|_1\frac{\kappa^2\|T_{a,x}\|_{\infty}}{1-\kappa^2\|T_{a,x}\|_{\infty}}\,,
	\end{aligned}
\end{equation*}
where we have used that $\|u^{\mathrm{in}}(x)\|_\infty \le 1$. We therefore obtain \eqref{gen_bound_1} provided $\alpha<1$.

Now, we turn to the proof of \eqref{gen_bound_2}. When $s>2$, obtaining a simple expression of the powers of $T_{a,x}$ appears difficult. However, the contribution of first order interactions between pairs of scatterers can be treated separately from higher order interactions. This is done as follows: 
\begin{equation*}
	\begin{aligned}
		|u^{\infty}(\hat{x},\theta)-u^{\infty,b}(\hat{x},\theta)|&= \frac{\kappa^2}{4\pi}\left|\sum\limits_{i=1}^s a_i e^{-i\kappa \hat{x}\cdot x_i}\left[((Id-\kappa^2 T_{a,x})^{-1}-Id)u^{\mathrm{in}}(x)\right]_i\right|\\
		&=\frac{\kappa^2}{4\pi}\left|\sum\limits_{i=1}^s a_i e^{-i\kappa \hat{x}\cdot x_i}\left[\left(\kappa^2T_{a,x}+\sum\limits_{n=2}^{+\infty}\kappa^{2n}T_{a,x}^n\right)u^{\mathrm{in}}(x)\right]_i\right|\\
		&\leq \frac{\kappa^2}{4\pi}\left|\sum\limits_{i=1}^s a_i e^{-i\kappa \hat{x}\cdot x_i}\left[\kappa^2T_{a,x}u^{\mathrm{in}}(x)\right]_i\right|+\frac{\kappa^2}{4\pi}\left|\sum\limits_{i=1}^s a_i e^{-i\kappa \hat{x}\cdot x_i}\left[\left(\sum\limits_{n=2}^{+\infty}\kappa^{2n}T_{a,x}^n\right)u^{\mathrm{in}}(x)\right]_i\right|\,.
	\end{aligned}
\end{equation*}
We bound the first term as follows:
\begin{equation*}
	\left|\sum\limits_{i=1}^s a_i e^{-i\kappa \hat{x}\cdot x_i}\left[\kappa^2T_{a,x}u^{\mathrm{in}}(x)\right]_i\right|=\kappa^2\Bigg|\sum\limits_{i=1}^s\sum\limits_{\substack{j=1\\j\neq i}}^s a_i a_j G(x_i,x_j)e^{-i\kappa(\hat{x}\cdot x_i-\theta\cdot x_j)}\Bigg|\leq 2\sum\limits_{i<j}|a_i| |a_j| \kappa^2\phi(|x_i-x_j|)\,,
\end{equation*}
and treat the second term as above to obtain
\begin{equation*}
	\left|\sum\limits_{i=1}^s a_i e^{-i\kappa \hat{x}\cdot x_i}\left[\left(\sum\limits_{n=2}^{+\infty}\kappa^{2n}T_{a,x}^n\right)u^{\mathrm{in}}\right]_i\right|\leq \|a\|_1\frac{\alpha^2}{1-\alpha}\,.
\end{equation*}
This finally yields \eqref{gen_bound_2}.
\end{proof}
We notice that, in the case $s=2$ and $a_1=a_2$, bounding $|G(x_1,x_2)|$ by $\phi(\Delta)$ in \eqref{gen_bound_1} yields~\eqref{two_bound}. However,~\eqref{two_bound} is significantly better when the two amplitudes are different. The second bound \eqref{gen_bound_2}, which explicitly depends on the strength of the pairwise interactions between the scatterers, allows us to bridge this gap. It is finer than \Cref{gen_bound_1}, since $\frac{\alpha}{1-\alpha}=\frac{\alpha^2}{1-\alpha}+\alpha$ and 
\begin{equation*}
	\begin{aligned}
	2\sum\limits_{i<j} |a_i| |a_j| \kappa^2\phi(|x_i-x_j|)&=\sum\limits_{i=1}^s |a_i| \sum\limits_{\substack{j=1\\j\neq i}}^s |a_j| \kappa^2\phi(|x_i-x_j|)\\
	&\leq \kappa^2\phi(\Delta)\sum\limits_{i=1}^s |a_i|\sum\limits_{\substack{j=1\\j\neq i}}^s |a_j|\\&\leq \kappa^2\phi(\Delta)\|a\|_1\,\underset{i}{\mathrm{max}}~\|a_{-i}\|_1=\alpha\|a\|_1\,.
	\end{aligned}
\end{equation*}

\paragraph{Numerical experiments.} We end this section by studying how tight  our bounds are in practice. We are especially interested in knowing whether the true error has the same dependence on $\kappa^2\phi(\Delta)$. Since~\eqref{gen_bound_1} and~\eqref{gen_bound_2} reduce to \eqref{two_bound} when $s=2$ and $a_1=a_2$, we focus on the study of \eqref{two_bound}. We take $d=2$ and consider~$100$ uniformly sampled incident and observation directions. We compute the linearization error $\|\Phi^f_x a-\Phi^b_x a\|_2$ where~$a=(1,1)$ and $x=(x_1,x_2)$ with $|x_1-x_2|=\Delta$ for multiple values of $\Delta$ and~$\kappa\in\{0.1,1,10\}$. For each value of $\Delta$, we average the errors corresponding to $20$ configurations obtained as follows: we draw~$x_1$ uniformly in $[-1,1]^2$, then draw $e$ uniformly in $\mathbb{S}^1$ and define $x_2\eqdef x_1+\Delta e$. The results are presented in \Cref{fig:lin_err}. We choose not to display error bars for the empirical error since we observe its standard deviation to be very small. We see that \eqref{two_bound} is almost tight when $\alpha$ is sufficiently bigger than~$1$, but deteriorates as $\alpha$ approaches~$1$.
\begin{figure}
	\hspace{2.8cm}
	\includegraphics[height=5.5cm]{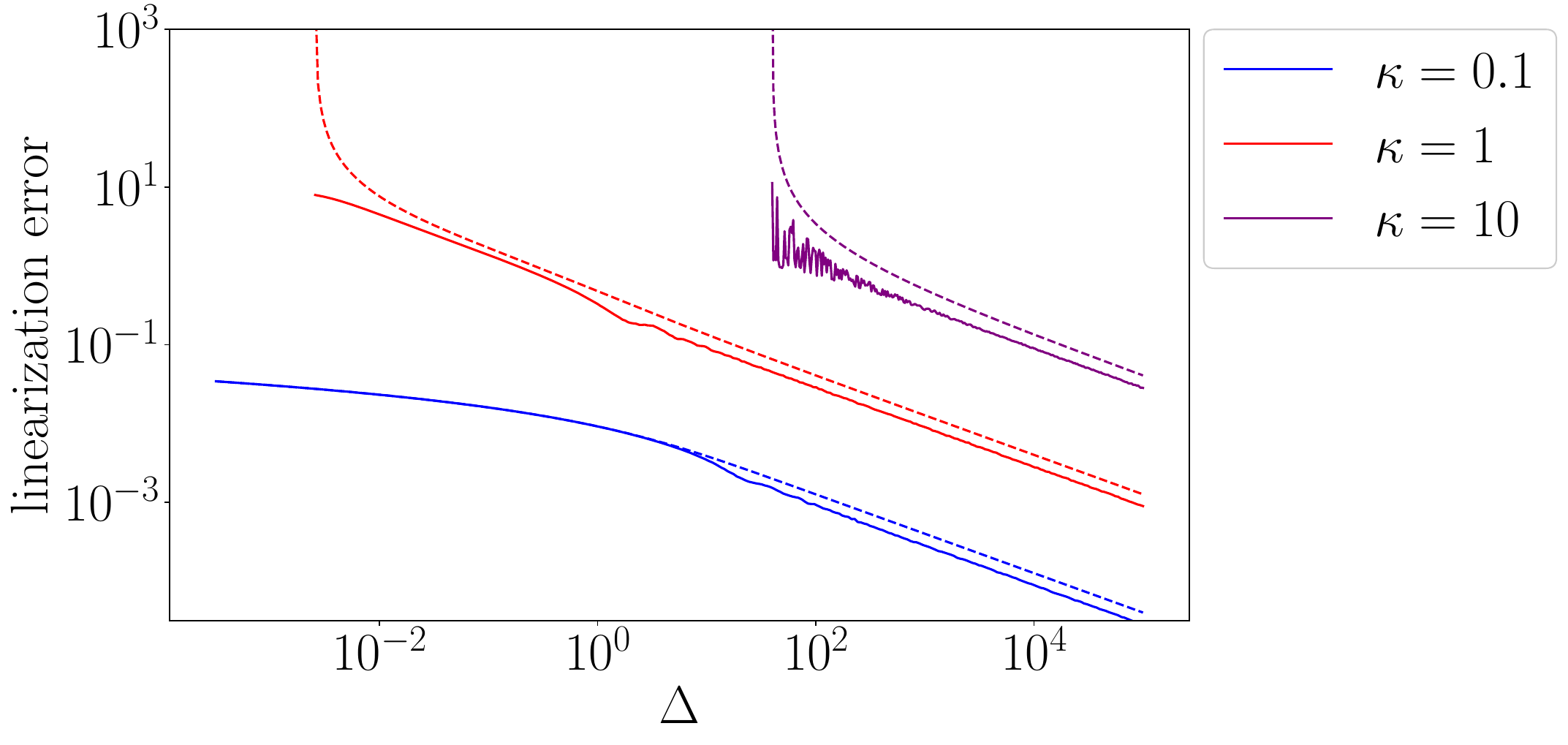}
	\caption{Dependance of the linearization error on $\kappa$ and $\Delta$. The dashed line corresponds to the bound \eqref{two_bound}, and the plain line to the empirical error.}
	\label{fig:lin_err}
\end{figure}
One could also wonder how \eqref{two_bound_prim} behaves outside the regime $|\beta|<1$. However, \Cref{fig:lin_err_bis} suggests that, in this case, $\|\Phi^f_x a-\Phi^b_x a\|_2$ is almost equal to $\|\Phi^f_x a\|_2+\|\Phi^b_x a\|_2$, so that even if \eqref{two_bound_prim} remains valid, this regime makes the linearization approach unsuitable for solving the inverse problem.
\begin{figure}
	\hspace{3cm}
	\includegraphics[height=5.5cm]{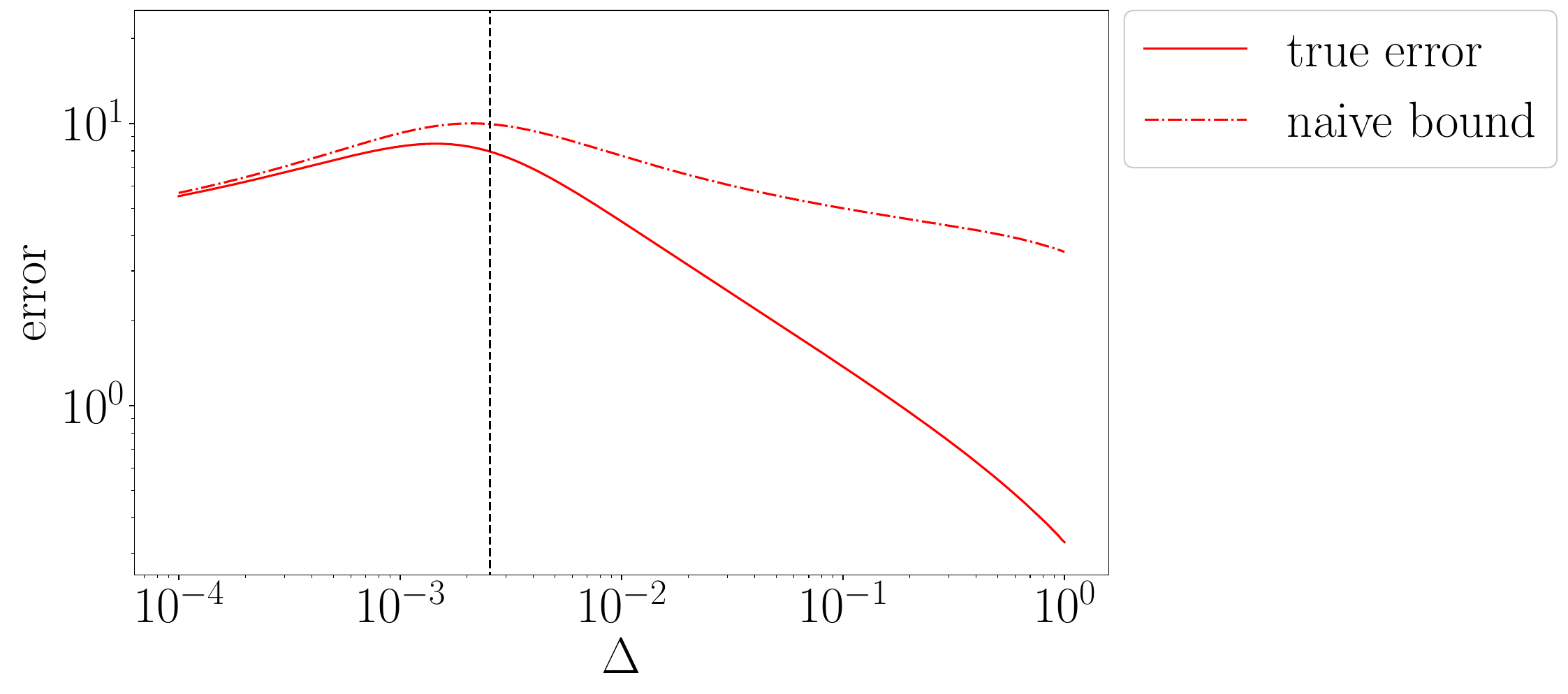}
	\caption{True linearization error $\|\Phi^f_x a-\Phi^b_x a\|_2$ and naive upper bound $\|\Phi^f_x a\|_2+\|\Phi^b_x a\|_2$ for $\kappa=1$. On the left of the vertical black line we have $|\beta|=\alpha>1$, and on the right $|\beta|=\alpha<1$.}
	\label{fig:lin_err_bis}
\end{figure}

\subsection{Stable reconstruction}
\label{sec_th_rec}
 We now provide guarantees that solutions of \cref{blasso} are close to $\mu$, provided a minimal separation condition holds and the sum $w$ of the linearization error and the noise is small. Our results are direct consequences of existing results on the sparse spikes problem (specifically, \cite[Theorem 3]{poonGeometryOfftheGridCompressed2023} and \cite[Theorem~1]{poonSupportLocalizationFisher2019}). We focus on the ``off-the-grid compressed sensing'' line of works \cite{tangCompressedSensingGrid2013,poonGeometryOfftheGridCompressed2023}, which exhibit conditions under which a discrete measure can be stably recovered from randomly sampled measurements. They consider forward operators of the form\footnote{Recall that $\mathcal{X}$ is a connected bounded open subset of $\RR^d$.}
\begin{equation}
	\begin{aligned}
		\mathcal{M}(\mathcal{X}) &  \to \mathbb{C}^m \\
		\mu & \mapsto \frac{1}{\sqrt{m}}\left[\int_{\mathcal{X}}\varphi_{\omega_k}\,d\mu\right]_{k=1}^m
	\end{aligned}
    \label{meas_op}
	\end{equation}
where $\omega_1,...,\omega_m$ are independent and identically distributed according to a certain probability distribution~$\Lambda$ and $\varphi_\omega$ are continuous functions on $\mathcal{X}$.

Recall that the far field pattern in the Born approximation \eqref{far_field_born} gives access to the Fourier transform of~$\mu=\sum_{i=1}^s a_i\delta_{x_i}$ in $\kappa\mathbb{S}^{d-1}-\kappa\mathbb{S}^{d-1}=\overline{B(0,2\kappa)}$, with $B(0,2\kappa)$ the open ball of radius $2\kappa$ centered at $0$. Thus, we take $\varphi_{\omega}\eqdef e^{-i\langle \omega,\cdot\rangle}$ and we let $\Lambda$ be the uniform law on~$B(0,2\kappa)$.

\paragraph{Sampling of the incident and observation directions.} Now, we explain how to sample the incident and observation directions in order to obtain uniform Fourier samples in $B(0,2\kappa)$. For simplicity, we consider the case $d=2$. From \cref{far_field_born}, we have that, for an incident direction $\theta\in\mathbb{S}^1$ and an observation direction $\hat{x}\in\mathbb{S}^1$, the far field pattern in the Born approximation is proportional to the Fourier transform of the unknown measure at the frequency~$\kappa(\hat{x}-\theta)$. We hence want to find $\pi_{(\hat{x},\theta)}$ such that, if~$(\hat{X},\Theta)\sim \pi_{(\hat{x},\theta)}$, then $\kappa(\hat{X}-\Theta)\sim \pi_{\omega}$ with $\pi_{\omega}=\mathcal{U}(B(0,2\kappa))$. In order to do this, let us identify $\RR^2$ with $\mathbb{C}$ and define
\begin{equation*}
	\begin{aligned}
		f \colon B(0,2\kappa) &  \to \mathbb{S}^1 \\
		\omega=re^{i\theta} & \mapsto e^{i\left(\theta+\mathrm{arcos}\left(\frac{r}{2\kappa}\right)\right)}\,,
	\end{aligned}
	~~~~~~~~
	\begin{aligned}
		g \colon B(0,2\kappa) &  \to \mathbb{S}^1 \\
		\omega=re^{i\theta} & \mapsto -e^{i\left(\theta-\mathrm{arcos}\left(\frac{r}{2\kappa}\right)\right)}\,.
	\end{aligned}
\end{equation*}
A simple computation shows that, for any $\omega\in B(0,2\kappa)$, we have $\omega=\kappa(f(\omega)-g(\omega))$. As a consequence, we obtain that taking $\pi_{\hat{x},\theta}=(f,g)_{\#}\pi_{\omega}$ yields $\kappa(\hat{X}-\Theta)\sim\pi_{\omega}$, where $(f,g)_{\#}\pi_{\omega}$ denotes the pushforward of $\pi_{\omega}$ by~$(f,g)$. In other words, one can sample $\omega_1,...,\omega_m$ uniformly in $B(0,2\kappa)$ and take~${\hat{x}_k=f(\omega_k)}$ and~$\theta_k=g(\omega_k)$ for all $k\in\{1,...,m\}$.

\paragraph{Stable recovery.} We first focus on the application of \cite[Theorem 3]{poonGeometryOfftheGridCompressed2023}. This result shows the stability of the reconstruction in terms of a partial optimal transport distance $\mathcal{T}_{\delta}$, which allows us to compare measures with different masses. It is defined, for any pair of positive Radon measures $\nu_1,\nu_2\in\mathcal{M}_+(\mathcal{X})$, by
\begin{equation*}
	\mathcal{T}_{\delta}(\nu_1,\nu_2)^2\eqdef \underset{\tilde{\nu}_1,\tilde{\nu}_2}{\mathrm{inf}}~\delta W_2^2(\tilde{\nu}_1,\tilde{\nu}_2)+|\nu_1-\tilde{\nu}_1|(\mathcal{X})+|\nu_2-\tilde{\nu}_2|(\mathcal{X})\,,
\end{equation*}
where $\delta\eqdef 2\kappa/\sqrt{d+2}\sim \kappa$, $W_2$ denotes the $2$-Wasserstein distance and the infimum is taken over all measures~$\tilde{\nu}_1,\tilde{\nu}_2$ satisfying ${|\tilde{\nu}_1|(\mathcal{X})=|\tilde{\nu}_2|(\mathcal{X})}$. We refer the reader to e.g.\ \cite{santambrogioOptimalTransportApplied2015,peyreComputationalOptimalTransport2019} for more details on optimal transport and its extensions.

\begin{theorem} Let $a\in(\mathbb{C}^*)^s$ and $x\in\mathcal{X}^s$. Assume that $(\hat{x}_k,\theta_k)_{1\leq k\leq m}$ are sampled so that $(\kappa(\hat{x}_k-\theta_k))_{1\leq k\leq m}$ are i.i.d. samples from $\mathcal{U}(B(0,2\kappa))$, and that~${y=(u^{\infty,b}(\hat{x}_k,\theta_k)+w_k)_{1\leq k\leq m}}$ with $w\in \mathbb{C}^m$. Then there exists~${r\sim 1/\kappa}$ such that, assuming
	$$m\gtrsim s\bigg(\mathrm{log}(s)\,\mathrm{log}\bigg(\frac{s}{\rho}\bigg)+\mathrm{log}\bigg(\frac{s\kappa}{\rho}\bigg)\bigg)$$
 for some $\rho\in (0,1)$ and
	$$\mathrm{min}_{i\neq j}|x_{i}-x_{j}|\gtrsim\frac{s^{2/(d+1)}}{\kappa}\,,$$
with probability at least $1-\rho$, every solution $\hat{\mu}$ of \cref{blasso2} with $\lambda^b\sim \|w\|_2/\sqrt{s}$ is~$\sqrt{s}\|w\|_2$-close to $\sum_{i=1}^s a_i\delta_{x_i}$, i.e.
\begin{equation}
	\mathcal{T}_{\delta}^2\left(\sum\limits_{i=1}^s \hat{A}_i \delta_{x_{i}},|\hat{\mu}|\right)\lesssim\sqrt{s}\|w\|_2~~\text{and}~~\underset{1\leq i\leq s}{\mathrm{max}}~|\hat{a}_i-a_i|\lesssim\sqrt{s}\,\|w\|_2\,,
	\label{recovery_guarantee}
\end{equation}
where $\hat{A}_i=|\hat{\mu}|(B(x_i,r))$ and $\hat{a}_i=\hat{\mu}(B(x_i,r))$.
\label{stable_recovery}
\end{theorem}
 Let us stress that \eqref{recovery_guarantee} tells that the support of $\hat{\mu}$ is essentially located on small neighborhoods of the true locations $\{x_i\}_{1\leq i\leq s}$, and that, in each of these neighborhoods, the average value $\hat{a}_i$ of $\hat{\mu}$ is close to the true amplitude $a_i$. The number of scalar measurements $m$ is, up to logarithmic factors, linearly  proportional to the number of spikes $s$, which is the best one can hope for.

\paragraph{Exact support recovery.} In their earlier work \cite{poonSupportLocalizationFisher2019}, the authors of \cite{poonGeometryOfftheGridCompressed2023} showed stronger \emph{exact support recovery} results under more restrictive assumptions. Their application to our problem is the object of \Cref{prop_exact_supp_rec}. Contrary to \Cref{stable_recovery}, this result does not hold for arbitrary $w$. It holds only if $\|w\|_2$ is sufficiently small, in terms of the sparsity $s$ and the intensities $(a_i)_{1\leq i\leq s}$, and with a larger number of measurements $m$.
\begin{theorem} Let $a\in(\mathbb{C}^*)^s$, $x\in\mathcal{X}^s$ and define $\underline{a}\eqdef\mathrm{min}_{1\leq i\leq s}\{|a_i|^2,|a_i|^{-2}\}$ and ${D\eqdef \underline{a}\,\mathrm{min}(\sqrt{s},\sqrt{s}/\|a\|_2,1)}$. Assume that $(\hat{x}_k,\theta_k)_{1\leq k\leq m}$ are sampled so that $(\kappa(\hat{x}_k-\theta_k))_{1\leq k\leq m}$ are i.i.d. samples from $\mathcal{U}(B(0,2\kappa))$, and that~${y=(u^{\infty,b}(\hat{x}_k,\theta_k)+w_k)_{1\leq k\leq m}}$ with $w\in \mathbb{C}^m$. If
	\begin{equation*}
		m\gtrsim s^{3/2}\,\log(\kappa/\rho),\qquad\lambda^b\lesssim D/s,\qquad \|w\|_2\lesssim \lambda^b\,,
	\end{equation*}
	for some $\rho\in(0,1)$ and
	$$\mathrm{min}_{i\neq j}|x_{i}-x_{j}|\gtrsim\frac{s^{2/(d+1)}}{\kappa}\,,$$
 then, with probability at least $1-\rho$, the solution $\hat{\mu}$ of \cref{blasso2} is unique and made of exactly $s$ spikes. Moreover, writing~$\hat{\mu}=\sum_{i=1}^s \hat{a}_i\delta_{\hat{x}_i}$, we have:
	\begin{equation*}
		\|\hat{a}-a\|_2+\kappa \|\hat{x}-x\|_2\lesssim \frac{\sqrt{s}(\lambda^b+\|w\|_2)}{\mathrm{min}_{1\leq i\leq s} |a_i|}\,.
	\end{equation*}
\label{prop_exact_supp_rec}
	\end{theorem}

The proof of \Cref{stable_recovery,{prop_exact_supp_rec}} is postponed to \Cref{appendix_kernel}. These results are a consequence of \cite[Theorem 3]{poonGeometryOfftheGridCompressed2023} and \cite[Theorem~1]{poonSupportLocalizationFisher2019}. We prove their applicability in the case of continuous Fourier sampling with frequencies uniformly distributed in a Euclidean ball of fixed radius.
 
We end this section with two remarks on the above results.
\begin{remark}
	Our bound on $\|u^{\infty}-u^{\infty,b}\|_{\infty}$ (see \Cref{sec_bound_lin_err}) yields a bound on $\|w^{\rm lin}\|_2$, which can in turn be used in \Cref{stable_recovery,prop_exact_supp_rec} by taking $w=w^{\rm lin}+w^{\rm noise}$. Indeed, we have
	\begin{equation*}
		\begin{aligned}
			\|w^{\rm lin}\|_2=\|\Phi^b_x a-\Phi_x^{f}a\|_2&=\sqrt{\sum\limits_{k=1}^m \left(\frac{1}{\sqrt{m}}u^{\infty}(\hat{x}_k,\theta_k)-\frac{1}{\sqrt{m}}u^{\infty,b}(\hat{x}_k,\theta_k)\right)^2}\\
			&=\sqrt{\frac{1}{m}\sum\limits_{k=1}^m \left|u^{\infty}(\hat{x}_k,\theta_k)-u^{\infty,b}(\hat{x}_k,\theta_k)\right|^2}\leq \|u^{\infty}-u^{\infty,b}\|_{\infty}\,.
		\end{aligned}
	\end{equation*}
\end{remark}

\begin{remark}
One could wonder how changing the sampling scheme $\Lambda$ would impact the results of \Cref{stable_recovery,prop_exact_supp_rec}. Owing to the proofs of \cite[Theorem 3]{poonGeometryOfftheGridCompressed2023} and \cite[Theorem 1]{poonSupportLocalizationFisher2019}, we only expect this choice to impact the minimal separation distance condition (at least for reasonable choices of sampling schemes). In fact, we believe that sampling the lower frequencies more often than the higher frequencies (instead of sampling them uniformly as above) would yield a kernel with a faster decay (see \Cref{appendix_kernel}). Given that~$d=2$ or $3$ in our case, this could allow for a minimal separation of the form $$\underset{i\neq j}{\mathrm{min}}~{|x_i-x_j|}\gtrsim \frac{1}{\kappa}\,,$$
without any dependence on the sparsity $s$ (see \cite[Section 2.1 and Lemma 17]{poonGeometryOfftheGridCompressed2023}).
\label{remark_sampling}
\end{remark}

\section{Reconstruction algorithm}
\label{sec_algo}
\paragraph{Description.} As discussed in Section~\ref{sec:overview}, the algorithm we propose to recover $\mu$ from $y$ is made of two main steps. The first one consists in solving \Cref{blasso2} using the sliding Frank-Wolfe algorithm introduced in \cite{denoyelleSlidingFrankWolfe2019}. The second one consists in locally minimizing~\Cref{blasso_obj_nonlin} using the output of the first step as initialization. This procedure is precisely described in Algorithm \ref{sfw}.

\begin{algorithm}[h!]
	\DontPrintSemicolon
  \KwData{measurements $y$, regularization parameters $(\lambda^b,\lambda^{f})$, stopping criterion tolerance $\epsilon$}
  \KwResult{estimated measure $\hat{\mu}$}
  $k\gets 0$\;
  $\mu^{[0]}\gets 0$\;
  \While{true}{
  $\eta^{[n]}\gets (\Phi^b)^*(y-\Phi^b \mu^{[n]})$\;
  $x_*\gets \underset{\substack{x\in\mathcal{X}}}{\text{Argmax}}~|\eta^{[n]}(x)|$~\tcp{new atom finding}
  \eIf{$|\eta^{[n]}(x_*)|\leq \lambda^b(1+\epsilon)$ and $n\geq 1$}{
	\tcp{nonlinear sliding step}
	find a critical point $(a^{*},x^{*})$ of $(a,x)\mapsto \frac{1}{2}\|\Phi^{f}_x a-y\|_2^2+\lambda^{f}\|a\|_1$ by performing a local descent initialized with~${(a^{[n]},x^{[n]})}$\;
	$\hat{\mu}\gets\sum_{i}a^{*}_i \delta_{x^{*}_i}$\;
	remove atoms with zero amplitude\;
  output $\hat{\mu}$\;}
  {$x^{[n+1/2]}\gets(x^{[n]},x_*)$\;
  $a^{[n+1/2]}\gets \underset{a}{\text{Argmin}}~\frac{1}{2}\|\Phi^b_{x^{[n+1/2]}}a-y\|_2^2+\lambda^b\|a\|_1$~\tcp{weights update}
  \tcp{linear sliding step}
  find a critical point $(a^{[n+1]},x^{[n+1]})$ of $(a,x)\mapsto \frac{1}{2}\|\Phi^b_x a-y\|_2^2+\lambda^b\|a\|_1$ by performing a local descent initialized with~${(a^{[n+1/2]},x^{[n+1/2]})}$\;
  $\mu^{[n+1]}\gets\sum_{i}a^{[n+1]}_i \delta_{x^{[n+1]}_i}$\;
  remove atoms with zero amplitude and merge spikes with same locations\;
  $n\gets n+1$\;
  }
  }
  \caption{Reconstruction of a discrete measure from far field measurements}
  \label{sfw}
  \end{algorithm}
  
\paragraph{Stopping condition and convergence of the linear step.} The condition of Line 6 is the standard stopping criterion for the sliding Frank-Wolfe algorithm (see \cite{denoyelleSlidingFrankWolfe2019}). In particular, it guarantees that the estimated measure is at most $\epsilon$-suboptimal in terms of the objective value for \eqref{blasso2}. By \cite[Lemma 5, Proposition~5]{denoyelleSlidingFrankWolfe2019}, we have that the sequence of iterates produced by the linear step (i.e.\ without performing Lines 6-10) converges in objective value towards the value of \Cref{blasso2} at a rate $O(1/n)$. Moreover, it has an accumulation point for the weak-* topology on $\mathcal{M}(\mathcal{X})$, which is a solution of \Cref{blasso2}. We also stress that, under a non-degeneracy assumption, Theorem 3 in the above reference yields the weak-* convergence of the sequence towards the unique minimizer of \Cref{blasso2} in a finite number of steps.

\paragraph{Implementation.} The implementation of Algorithm \ref{sfw} requires several oracles to carry out the operations of Lines 5, 7, 13 and~14. We refer the reader to \cite[Remark 9]{denoyelleSlidingFrankWolfe2019} regarding the operations of Lines 5, 13 and~14. The local minimization of Line 7 is performed as in Line 14. To solve the various optimization problems appearing in Algorithm \ref{sfw}, we rely on the Python libraries JAX \cite{bradburyJAXComposableTransformations2018} and JAXopt \cite{blondelEfficientModularImplicit2022}.

\paragraph{Hyperparameters.} Algorithm \ref{sfw} involves several hyperparameters. The main ones are the regularization parameters for the linear and nonlinear steps $\lambda^b$ and $\lambda^f$, as well as the stopping criterion tolerance $\epsilon$. When carrying out Lines 9 and 16, we also need to define two thresholds: one below which the amplitudes are set to zero, and one below which neighboring spikes are merged\footnote{We stress that, due to the linearity of $\Phi^b$ and the nonlinearity of $\Phi^f$, we merge neighboring spikes only in the linear step.}. The most critical parameters are $\lambda^b$ and $\lambda^f$. In the literature on variational regularization, they are usually selected after performing a grid search, or using more advanced techniques under statistical assumptions on the noise. We leave the design of rules for automatically choosing them to future works, and provide in \Cref{sec_num_res} the set of parameters used for each experiment. Let us point out that $\lambda^b$ should be chosen according to an estimate of the linearization error, which could be difficult to know a priori in practice.

\section{Numerical results}
\label{sec_num_res}
In this section, we investigate the performance of Algorithm \ref{sfw}. We consider the two-dimensional case. For visualization purposes, we restrict ourselves to the case of real amplitudes, i.e.\ $a\in\RR^s$. All the experiments below can be reproduced using the code available at~\url{https://github.com/rpetit/pointscat}. The domain~$\mathcal{X}$ is taken to be $(-r/2,r/2)^2$ with $r=5$ in the first setting and $r=10$ in the second one.
\begin{figure}
	\subfigure[No measurement noise.]{\includegraphics[width=\textwidth]{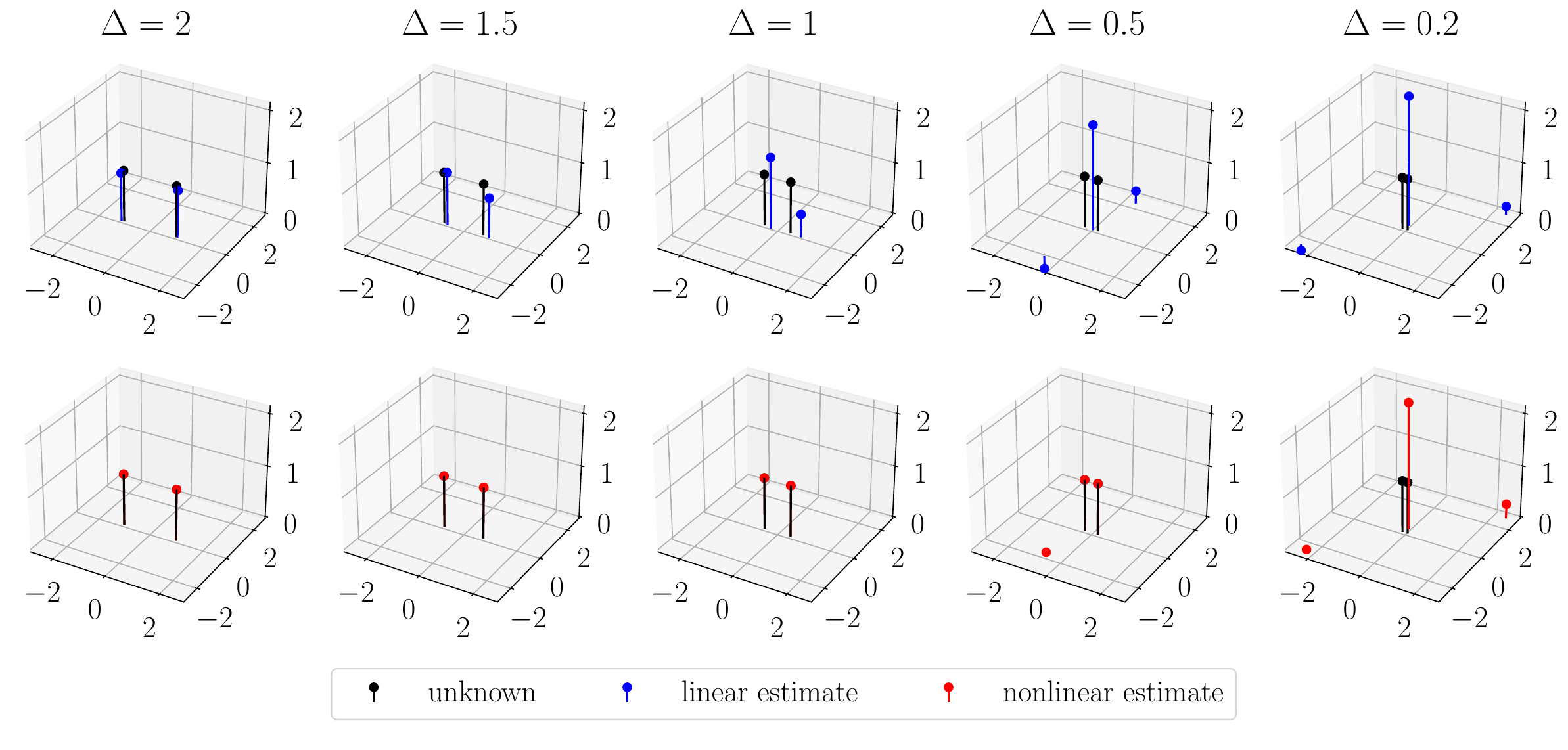}\label{fig:exp_1_1}}
	\subfigure[Gaussian measurement noise with standard deviation $0.1$.]{\includegraphics[width=\textwidth]{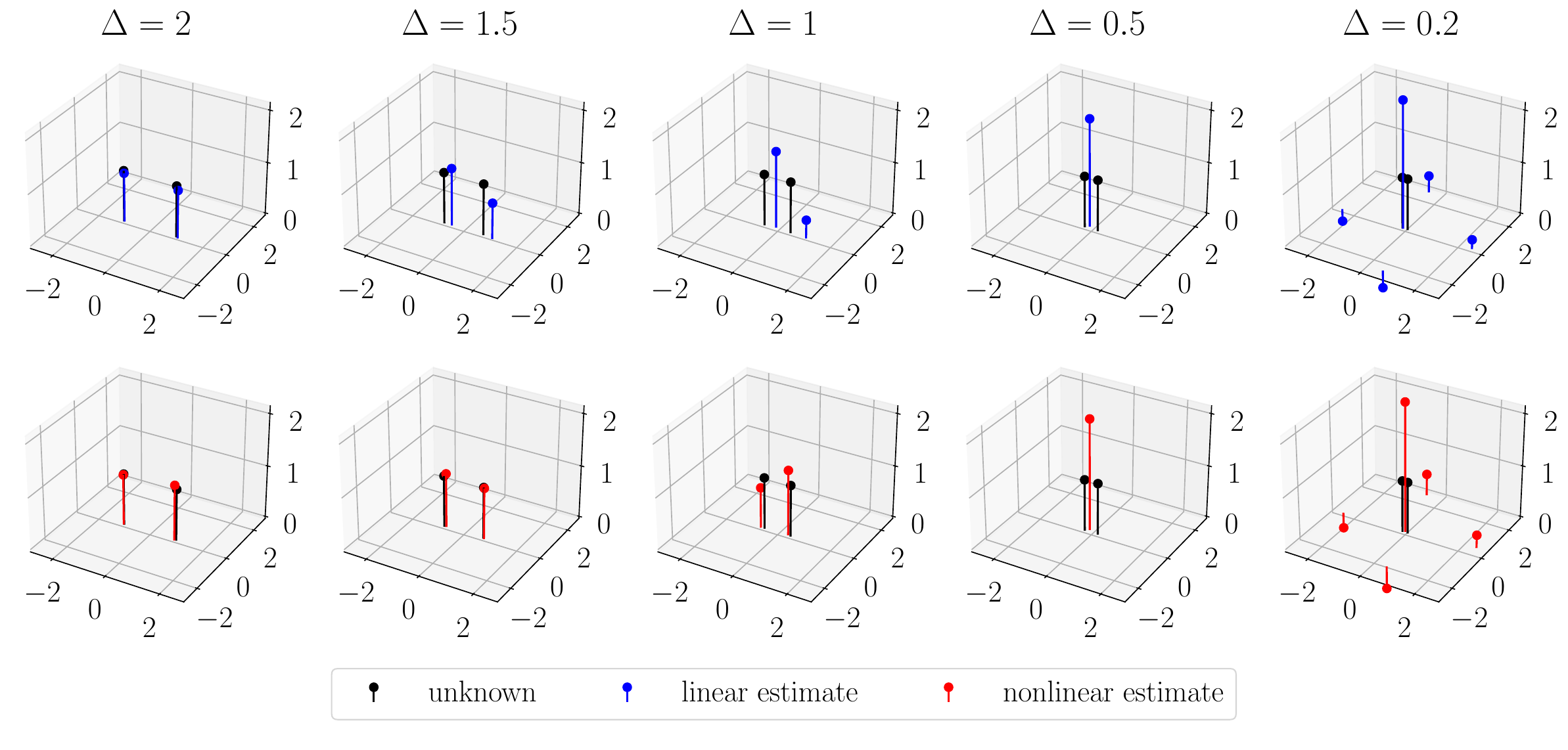}\label{fig:exp_1_2}}
	\caption{Linear and nonlinear estimates for two unknown scatterers with several separation distances~$\Delta$, obtained with $m=20$ measurements.}\label{fig:exp_1}
   \end{figure}

\paragraph{Two scatterers.} In \Cref{fig:exp_1}, we present the output of Algorithm \ref{sfw} in a first simple setting. We refer to the output of Algorithm \ref{sfw} without performing Lines 6-10 as the \emph{linear estimate} and to the output of the full algorithm as the \emph{nonlinear estimate}. The unknown measure is $\mu=\delta_{x_1}+\delta_{x_2}$ with $x_1=(-\Delta/2, 0)$ and~$x_2=(\Delta/2,0)$, the wave number is~${\kappa=1}$, and the number of randomly sampled frequencies is $m=20$. In \Cref{fig:exp_1_1}, there is no noise on the observations and we fix $\lambda^b=0.5$ and $\lambda^f=10^{-3}$. In \Cref{fig:exp_1_2}, the measurement noise is Gaussian and has standard deviation~$0.1$~(corresponding to a relative $\ell_2$ noise level of~$11\%,11\%,10\%,8\%,6\%$ respectively) and we fix~${\lambda^b=1.0}$ and $\lambda^f=0.1$. We observe that, in both cases, the quality of the linear and nonlinear estimates deteriorates as $\Delta$ decreases. This is expected from the results of \Cref{sec:rec}: a smaller $\Delta$ calls for a larger $\kappa$ (\Cref{stable_recovery,prop_exact_supp_rec}),  yielding in turn a larger linearization error (\Cref{sec_bound_lin_err}), also because $\phi$ is non-increasing. In some cases, the nonlinear step allows us to significantly improve the quality of the reconstruction. For very close scatterers, neither the linear nor the nonlinear estimation allow us to accurately recover the unknown measure.

\paragraph{A more challenging example.} In \Cref{fig:exp_2}, we show the output of Algorithm \ref{sfw} for a measure $\mu$ made of~$9$ spikes and $m=100$ sampled frequencies. The measurement noise is Gaussian and has standard deviation~$0.1$~(corresponding to a relative $\ell_2$ noise level of $6\%$). This setting is already highly nonlinear, since the relative $\ell_2$ error between the noiseless measurements and those corresponding to the Born approximation is of $24\%$. Yet, the linear estimate is close to the unknown measure and the nonlinear step allows further improvement in the quality of the reconstruction.

\begin{figure}[h!]
	\centering
	\includegraphics[width=.8\textwidth]{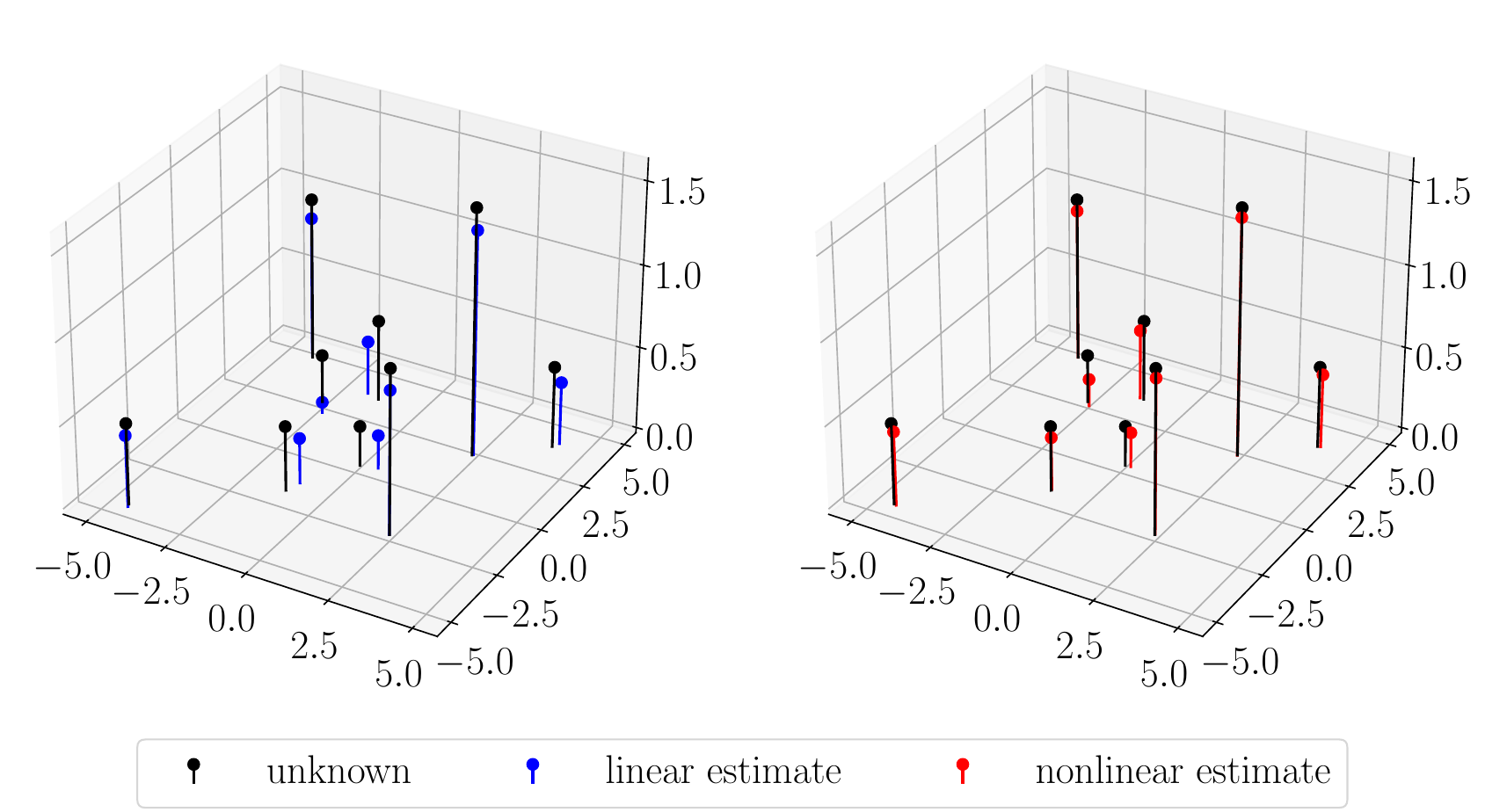}
	\caption{Linear and nonlinear estimates for nine unknown scatterers with $m=100$ measurements and Gaussian measurement noise with standard deviation $0.1$.}
	\label{fig:exp_2}
\end{figure}

\paragraph{Influence of the initialization of the nonlinear step.} As a final experiment, we investigate the outcome of the nonlinear step if, instead of being initialized with the linear estimate, we initialize it with a discrete measure whose support is a uniform discretization of the domain. We use the same experimental setup as in \Cref{fig:exp_2}. In \Cref{fig:exp_3}, we show that even a rather fine initialization might fail to recover one of the scatterers. Still, increasing the number of initial spikes allows for an accurate recovery. This experiment suggests to further analyze the nonlinear dynamic, and in particular the existence of spurious critical points.

\begin{figure}
	\centering
	\includegraphics[width=.8\textwidth]{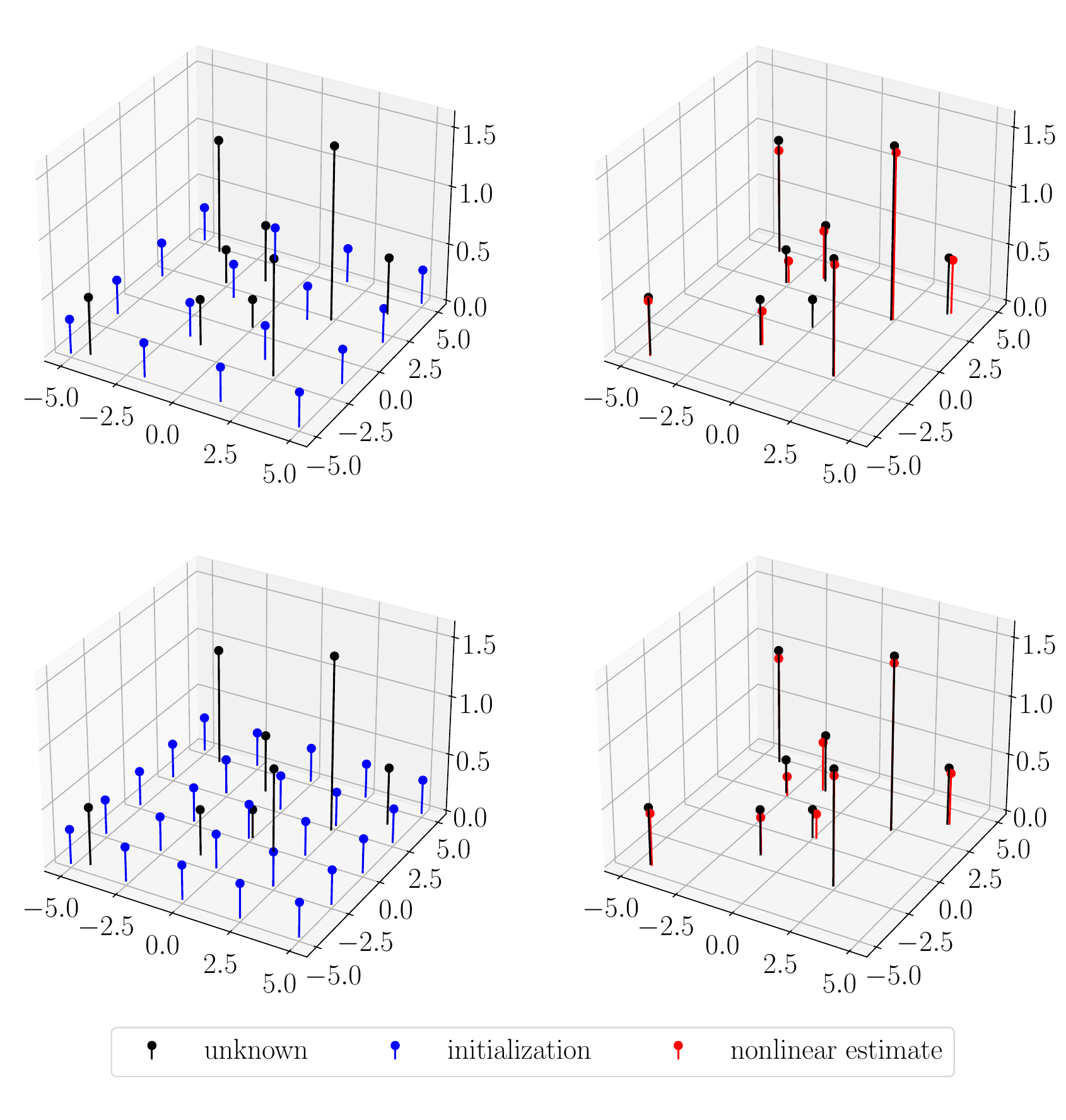}
	\caption{Output of the nonlinear step initialized with a measure supported on a $4\times 4$ (top row) and a $5\times 5$ (bottom row) regular grid. The experimental setup is the same as in \Cref{fig:exp_2}.}
	\label{fig:exp_3}
\end{figure}

\section{Conclusion}
In this paper, we proposed to connect the inverse scattering problem associated to pointwise inhomogeneities with the sparse spikes problem. We leveraged existing results on the sparse spikes problem to obtain guarantees for a reconstruction method involving the linearized forward operator. We showed that the error is controlled by the strength of the scatterers and the minimal separation distance between them. Our analysis suggests the existence of a tradeoff in the choice of the incident frequency: a low frequency allows us to reduce the linearization error but imposes a larger separation distance between the scatterers. Finally, we proposed a refinement step involving the nonlinear forward operator, and provided numerical evidence of improvements in the reconstructions with respect to the linear estimation procedure.

A natural question that we left open is to understand when the nonlinear step initialized with the output of the linear step allows us to recover the unknown accurately. One could start by investigating a simpler setting, where the scatterers are constrained to lie on a uniform discretization of the domain.

Future works could also include taking into account higher order terms in the first estimation step, for example by using the lifting and relaxation strategies developed for quadratic inverse problems (see for instance \cite{candesPhaseLiftExactStable2013}). This would yield an estimation error of the form $\alpha^2/(1-\alpha)$ instead of $\alpha/(1-\alpha)$ as in \Cref{gen_bound_1}, but would not allow to lift the requirement $\alpha<1$. Another possible direction could be to rely on consecutive approximations. By inserting the estimates of $a$ and $x$ obtained in the linear step in the system~\eqref{foldy}, its solution $u$ (instead of $u^{in}(x)$ in the Born approximation) could be in turn inserted in \eqref{far_field}, preserving the linearity of the forward map. Repeating this procedure could lead to more accurate estimates.

Finally, one could also wish to consider a more realistic measurement acquisition scenario, i.e.\ observations of the form $u^{\infty}(\hat{x}_{k_1},\theta_{k_2})$ with $1\leq k_1\leq m_1$ and $1\leq k_2\leq m_2$ instead of $u^{\infty}(\hat{x}_k,\theta_k)$ with $1\leq k\leq m$ \cite{gilbertNonlinearIterativeHard2020,albertiInfiniteDimensionalInverseProblems2022}. This would correspond to fixing the observation directions (and hence the set of sensing devices used), and, for every incident wave sent in the medium, acquire a full set of measurements. Our reconstruction algorithm applies to this setting, but the results of \Cref{sec_th_rec} correspond instead to acquiring a single measurement for every incident wave. We stress that treating the former scenario would require the development of different tools, as the corresponding frequencies would no longer be independent and identically distributed. In a similar spirit, it would be interesting to investigate the problem when the Foldy-Lax system is replaced by more advanced models for the scattering by small inhomogeneities \cite{carminatiPrinciplesScatteringTransport2021}.

\section*{Acknowledgments}
The authors thank Mourad Sini for valuable discussions about this work. Co-funded by the European Union (ERC, SAMPDE, 101041040). Views and opinions
expressed are however those of the authors only and do not necessarily reflect
those of the European Union or the European Research Council. Neither the
European Union nor the granting authority can be held responsible for them. GSA and MS are members of the ``Gruppo Nazionale per l’Analisi Matematica, la Probabilità e le loro Applicazioni'', of the ``Istituto Nazionale di Alta Matematica''. The research was supported in part by the MIUR Excellence Department Project awarded to Dipartimento di Matematica, Università di Genova, CUP D33C23001110001. Funded by European Union – Next Generation EU.

\bibliography{ref}
\bibliographystyle{plain}

\appendix
\section{Proof of Theorems \ref{stable_recovery} and \ref{prop_exact_supp_rec}}
\label{appendix_kernel}
In this section, we prove \Cref{stable_recovery,prop_exact_supp_rec} by showing that the assumptions of \cite[Theorem 3]{poonGeometryOfftheGridCompressed2023} and \cite[Theorem 1]{poonSupportLocalizationFisher2019} are satisfied in our setting. This amounts to studying the properties of a kernel $K\colon\mathcal{X}\times\mathcal{X}\to\RR$ associated to the measurement operator. For a forward operator of the form \eqref{meas_op}, $K$ is defined by
\begin{equation}
    K(x,x')\eqdef \int_{\mathcal{X}} \overline{\varphi_{\omega}(x)}\varphi_{\omega}(x') d\Lambda(\omega),\quad(x,x')\in\mathcal{X}\times\mathcal{X}\,.
    \label{def_kernel}
\end{equation}
We first begin by deriving the expression of $K$ in our setting, and then turn to the study of its properties.
\begin{proposition}
    If $\varphi_{\omega}=e^{-i\langle \omega,\cdot\rangle}$ and $\Lambda$ is the uniform law on $B(0,2\kappa)$, then for every $x,x'\in\mathcal{X}$ we have~$K(x,x')=\rho(\sigma^{-1}|x-x'|)$ with $\sigma^{-1}\eqdef 2\kappa$ and $\rho(s)\eqdef (2/s)^{d/2} \Gamma(d/2+1) J_{d/2}(s)$, where $J_{\alpha}$ denotes the Bessel function of the first kind and order $\alpha$.
\end{proposition}
\begin{proof}
    A simple computation shows that $K(x,x')$ is proportional to the Fourier transform of $\mathbf{1}_{B(0,\sigma^{-1})}$ at~$x-x'$, which allows us to conclude.
\end{proof}
We notice that, since $J_{\alpha}(s)\sim(s/2)^{\alpha}/\Gamma(\alpha+1)$ at $0$, we have $K(x,x)=\mathbb{E}_{\omega\sim\Lambda}[|\varphi_{\omega}(x)|^2]=\rho(0)=1$ for every $x\in\mathcal{X}$. For convenience, we define $\rho_{\sigma}\eqdef \rho(\sigma^{-1}\cdot)$. Let us now list a few notations and definitions introduced in~\cite{poonSupportLocalizationFisher2019,poonGeometryOfftheGridCompressed2023}, before stating the various results we need to prove. 
\paragraph{The Fisher metric.} In \cite{poonSupportLocalizationFisher2019,poonGeometryOfftheGridCompressed2023}, a special metric on~$\mathcal{X}$ related to $K$ is introduced. In our case, the metric tensor $\mathfrak{g}_x\eqdef\nabla_1\nabla_2 K(x,x)$ is constant and $\mathfrak{g}_x=-\rho_{\sigma}''(0)\,Id=-\sigma^{-2}\rho''(0)\,Id$ with~$\rho''(0)=-1/(d+2)$. Hence, the associated distance is given by 
\begin{equation*}
    \mathfrak{d}_{\mathfrak{g}}(x,x')=\sqrt{|\rho_{\sigma}''(0)|}\,|x-x'|=\sqrt{-\rho''(0)}\,\sigma^{-1}|x-x'|,\qquad x,x'\in\mathcal{X}.
\end{equation*}
For every $u\in\mathbb{C}^d$, we define $|u|_x\eqdef \sqrt{u^*\mathfrak{g}_x u}=\sigma^{-1}\sqrt{-\rho''(0)}\,|u|$. 

\paragraph{Covariant derivatives.} For $j\in\{0,1,2\}$, the ``covariant derivative'' $D_j[f](x)\colon(\mathbb{C}^d)^j\to\mathbb{C}$ are defined as
\begin{equation*}
        D_0[f](x)\eqdef f(x),~~~D_1[f](x)[v]\eqdef v^T\nabla f(x),~~~D_2[f](x)[v_1,v_2]\eqdef v_1^T \nabla^2 f(x) v_2.
\end{equation*} The associated operator norms are defined by 
\begin{equation*}
 \|D_1[f]\|_x\eqdef \|\mathfrak{g}_x^{-1/2}\nabla f(x)\|_2,~~~\|D_2[f]\|_x\eqdef \|\mathfrak{g}_x^{-1/2}\nabla^2 f(x)\mathfrak{g}_x^{-1/2}\|_2,
\end{equation*} 
and we also define $L_0(\omega)\eqdef \|\varphi_{\omega}\|_{\infty}$, $L_j(\omega)\eqdef\mathrm{sup}_{x\in\mathcal{X}}\,\|D_j[\varphi_{\omega}]\|_x,~(j=1,2)$. Given $0\leq i,j\leq 2$, let~$K^{(ij)}(x,x')$ be defined, for~$Q\in(\mathbb{C}^d)^i$ and $V\in(\mathbb{C}^d)^j$, by
\begin{equation*}
    [Q]K^{(ij)}(x,x')[V]\eqdef \mathbb{E}\Big[\overline{D_i[\varphi_{\omega}](x)[Q]}D_j[\varphi_{\omega}](x')[V]\Big].
\end{equation*}
The associated operator norms are defined by
\begin{equation}
    \left\|K^{(ij)}(x,x')\right\|_{x,x'}\eqdef\underset{Q,V}{\mathrm{sup}}~[Q]K^{(ij)}(x,x')[V],
    \label{norm_cov_deriv_kernel}
\end{equation}
where the supremum is taken over all $Q=[q_1,...,q_i]$ and $V=[v_1,...,v_j]$ such that $|q_k|_x\leq 1$ for every~${k\in\{1,...,i\}}$ and $|v_l|_{x'}\leq 1$ for every $l\in\{1,...,j\}$. In \eqref{norm_cov_deriv_kernel}, we rather write $\|\cdot\|_x$ instead of $\|\cdot \|_{x,x'}$ when the dependance is only on $x$ (that is to say when $j=0$). We also define $B_{ij}\eqdef \mathrm{sup}_{x,x'\in\mathcal{X}}\,\|K^{(ij)}(x,x')\|_{x,x'}$.

\paragraph{Kernel width.} Finally, we define the kernel width of $K$, defined in \cite[Definition 3]{poonGeometryOfftheGridCompressed2023} as:
\begin{equation*}W(h,s)\eqdef \mathrm{inf}\left\{\Delta\,\Bigg\rvert\,\sum_{k=2}^s \|K^{(ij)}(x_1,x_k)\|_{x_1,x_k}\leq h,~(i,j)\in\{0,1\}\times\{0,2\},~(x_k)_{k=1}^s\in\mathcal{S}_{\Delta}\right\}
\end{equation*}
where $\mathcal{S}_{\Delta}\eqdef \{(x_k)_{k=1}^s\in\mathcal{X}^s\,\rvert\,\forall k\neq l,~\mathfrak{d}_{\mathfrak{g}}(x_k,x_l)\geq \Delta\}$.

\paragraph{Outline of the appendix.} First, we state useful properties of Bessel functions in \Cref{sec_bessel}. In \Cref{sec_deriv_kernel}, we give the expressions of the derivatives of the kernel and the operator norm of its covariant derivatives. The remaining of the appendix is dedicated to the proof that assumptions (i)\,-\,(v) below hold.
\begin{enumerate}[label=(\roman*)]
    \item $B_{ij}=\mathcal{O}(1)$ for every $(i,j)\in\{0,1,2\}$.
    \item There exist $r_{\rm near}>0$ and $0<\overline{\epsilon}_2<r_{\rm near}^{-2}$ such that $-K^{(02)}(x,x')[v,v]\geq \overline{\epsilon}_2|v|_x^2$ for every~$v\in\mathbb{C}^d$ as soon as $\mathfrak{d}_{\mathfrak{g}}(x,x')\leq r_{\rm near}$.
    \item There exist $\overline{\epsilon}_0<1$ such that $|K(x,x')|\leq 1-\overline{\epsilon}_0$ as soon as $\mathfrak{d}_{\mathfrak{g}}(x,x')\geq r_{\rm near}$.
    \item $W(h,s)=\mathcal{O}(s^{2/(d+1)})$ as soon as $h=\mathcal{O}(1)$.
    \item $\|L_j\|_{\infty}=\mathcal{O}(1)$ for every $j\in\{0,...,3\}$.\footnote{See \cite[Section 5.1]{poonGeometryOfftheGridCompressed2023} for the definition of $L_3$.}
\end{enumerate}
We stress that the constants appearing in the various bounds do not depend on $\sigma$. Provided~${\text{(i)\,-\,(v)}}$ hold, \cite[Theorem 3]{poonGeometryOfftheGridCompressed2023} and \cite[Theorem~1]{poonSupportLocalizationFisher2019} can be applied to our setting, which yields precisely \Cref{stable_recovery,prop_exact_supp_rec}.
\subsection{Bessel functions}
\label{sec_bessel}
In this subsection, we give two properties of Bessel functions that we use in the following. First, we have
\begin{equation}
    \forall \alpha>0,~\forall x>0,~J'_{\alpha}(x)-(\alpha/x)J_{\alpha}(x)=-J_{\alpha+1}(x)\,.
    \label{bessel_derivative}
\end{equation}
Then, we have the following bound, which implies the boundedness of $x\mapsto |J_{\alpha}(x)|/x^{\beta}$ for any~$\beta\leq\alpha$.
\begin{lemma}
	For every $\alpha>0$ and $x>0$, we have: 
	\begin{equation*}
		|J_{\alpha}(x)|\leq \mathrm{min}\left(\frac{1}{\Gamma(\alpha+1)}\left(\frac{x}{2}\right)^{\alpha}\left(\left|1-\frac{1}{\alpha+1}\frac{x^2}{4}\right|+\frac{1}{(\alpha+1)(\alpha+2)}\left(e^{x^2/4}-1-\frac{x^2}{4}\right)\right),\frac{0.8}{x^{1/3}}\right).
	\end{equation*}
	\label{lemma_bound_bessel}
	\end{lemma}
	\begin{proof}
		The inequality $|J_{\alpha}(x)|\leq 0.8 x^{-1/3}$ is proved in \cite{landauBesselFunctionsMonotonicity2000}. For the other inequality, we use the series expansion of $J_{\alpha}$:
		\begin{equation*}
			J_{\alpha}(x)=\frac{1}{\Gamma(\alpha+1)}\left(\frac{x}{2}\right)^{\alpha}\left(1-\frac{1}{\alpha+1}\frac{x^2}{4}+\Gamma(\alpha+1)\sum\limits_{p=2}^{+\infty} \frac{(-1)^p}{p!\,\Gamma(p+\alpha+1)}\left(\frac{x}{2}\right)^{2p}\right).
		\end{equation*}
	Then, we notice that 
	\begin{equation*}
		\begin{aligned}
			\left|\Gamma(\alpha+1)\sum\limits_{p=2}^{+\infty} \frac{(-1)^p}{p!\,\Gamma(p+\alpha+1)}\left(\frac{x}{2}\right)^{2p}\right|&\leq \Gamma(\alpha+1)\sum\limits_{p=2}^{+\infty}\frac{1}{p!\,\Gamma(p+\alpha+1)}\left(\frac{x}{2}\right)^{2p}\\
			&\leq \frac{\Gamma(\alpha+1)}{\Gamma(\alpha+3)}\sum\limits_{p=2}^{+\infty}\frac{1}{p!}\left(\frac{x^2}{4}\right)^{p}\\
			&=\frac{1}{(\alpha+1)(\alpha+2)}\left(e^{x^2/4}-1-\frac{x^2}{4}\right),
		\end{aligned}
	\end{equation*}
	which yields the result.
	\end{proof}

\subsection{Derivatives of the kernel}
\label{sec_deriv_kernel}
Let $x,x'\in\mathcal{X}$. We define $t=x-x'$ and, abusing notation, denote~$K(t)=\rho_{\sigma}(|t|)$. We have the following:
\begin{equation*}
	\begin{aligned}
		\nabla K(t)=&\frac{\rho_{\sigma}'(|t|)}{|t|}\,t,\\
		\nabla^2 K(t)=&\frac{\rho_{\sigma}'(|t|)}{|t|}\,Id+\left[\rho_{\sigma}''(|t|)-\frac{\rho_{\sigma}'(|t|)}{|t|}\right]\frac{tt^T}{|t|^2}\\
		=&\frac{\rho_{\sigma}'(|t|)}{|t|}\left(Id-\frac{tt^T}{|t|^2}\right)+\rho''_{\sigma}(|t|)\frac{tt^T}{|t|^2},\\
		\partial_i\nabla^2 K(t)=&~~~\left[\rho_{\sigma}''(|t|)-\frac{\rho_{\sigma}'(|t|)}{|t|}\right]\frac{t_i}{|t|^2}Id  + \left[\rho_{\sigma}'''(|t|)\frac{t_i}{|t|}-\frac{t_i}{|t|^2}\left(\rho''_{\sigma}(|t|)-\frac{\rho'_{\sigma}(|t|)}{|t|}\right)\right]\frac{tt^T}{|t|^2}\\
		&+\left[\rho_{\sigma}''(|t|)-\frac{\rho_{\sigma}'(|t|)}{|t|}\right]\frac{1}{|t|^2}\left(te_i^T+e_it^T-2t_i\frac{tt^T}{|t|^2}\right)\\
		=&~~~\left[\rho_{\sigma}''(|t|)-\frac{\rho_{\sigma}'(|t|)}{|t|}\right]\frac{1}{|t|^2}\left(te_i^T+e_it^T+t_i\left(Id-3\frac{tt^T}{|t|^2}\right)\right)+\rho'''_{\sigma}(|t|)\frac{t_i}{|t|}\frac{tt^T}{|t|^2}\,.\\
	\end{aligned}
\end{equation*}
Considering the definition of \eqref{norm_cov_deriv_kernel}, we also obtain:
\begin{align}
		&\|K^{(00)}(x,x')\|=|K(t)|=|\rho(\sigma^{-1}|t|)|\,,\label{k_00}\\
		&\|K^{(10)}(x,x')\|_{x}=\frac{\sigma}{\sqrt{-\rho''(0)}} \|\nabla K(t)\|_2=\frac{1}{\sqrt{-\rho''(0)}}|\rho'(\sigma^{-1}|t|)|\,,\label{k_10}\\
		&\|K^{(11)}(x,x')\|_{x,x'}=\frac{\sigma^2}{(-\rho''(0))} \|\nabla^2 K(t)\|_2\leq  \frac{1}{(-\rho''(0))}\left(\left|\frac{\rho'(\sigma^{-1}|t|)}{\sigma^{-1}|t|}\right|+\left|\rho''(\sigma^{-1}|t|)-\frac{\rho'(\sigma^{-1}|t|)}{\sigma^{-1}|t|}\right|\right),\label{k_11}\\
		&\|K^{(20)}(x,x')\|_x=\|K^{(11)}(x,x')\|_{x,x'}\,.\label{k_20}
\end{align}

\subsection{Uniform bounds}
\begin{lemma}
	For $(i,j)\in\{0,1,2\}$, we have $B_{ij}=\mathcal{O}(1)$.
	\label{lemma_uniform_bounds}
\end{lemma}
\begin{proof}
	~

	\begin{itemize}
		\item $\bm{B_{00}}$: considering \eqref{k_00}, we have that $B_{00}\leq \|\rho\|_{\infty}$, which is finite by \Cref{lemma_bound_bessel}.
		\item $\bm{B_{10}}$: considering \eqref{k_10}, we have 
		$B_{10}\leq \sqrt{d+2}\,\|\rho'\|_{\infty}\leq \sqrt{5}\,\|\rho'\|_{\infty}$. Using \eqref{bessel_derivative}, we find 
  \begin{equation}
	\begin{aligned}
  \rho'(s)&=2^{d/2}\Gamma(d/2+1)(-(d/2)J_{d/2}(s)/s^{d/2+1}+J_{d/2}'(s)/s^{d/2})\\
  &=(2/s)^{d/2}\Gamma(d/2+1)(-(d/2)J_{d/2}(s)/s+(d/2)J_{d/2}(s)/s-J_{d/2+1}(s))\\
  &=(2/s)^{d/2}\Gamma(d/2+1)(-J_{d/2+1}(s))\,,
	\end{aligned}
	\label{rho_prime}
  \end{equation}
  and we can use \Cref{lemma_bound_bessel} again to get $B_{10}=\mathcal{O}(1)$.
		\item $\bm{B_{11}=B_{20}}$: considering \eqref{k_11} and \eqref{k_20}, we have
		\begin{equation*}
			B_{11}=B_{20}\leq (d+2)\left(\|\rho''\|_{\infty}+2~\underset{s> 0}{\mathrm{sup}}\,\left|\frac{\rho'(s)}{s}\right|\right).
		\end{equation*}
		Using \eqref{rho_prime} and \eqref{bessel_derivative} we get 
		\begin{equation}
			\begin{aligned}
			\rho''(s)&=2^{d/2}\Gamma(d/2+1)((d/2)J_{d/2+1}(s)/s^{d/2+1}-J_{d/2+1}'(s)/s^{d/2})\\
			&=(2/s)^{d/2}\Gamma(d/2+1)((d/2)J_{d/2+1}(s)/s-(d/2+1)J_{d/2+1}(s)/s+J_{d/2+2}(s))\\
			&=(2/s)^{d/2}\Gamma(d/2+1)(J_{d/2+2}(s)-J_{d/2+1}(s)/s)
			\end{aligned}
			\label{rho_second}
		\end{equation} 
		and another application of \Cref{lemma_bound_bessel} yields that $B_{11}=B_{20}=\mathcal{O}(1)$.
		\item $\bm{B_{21}}$: for any $x,x'\in\mathcal{X}$, defining $t=x-x'$, for every $q\in\mathbb{C}^d$, we have
		\begin{equation*}
			\begin{aligned}
			[q]K^{(12)}(x,x')&=\sum\limits_{i=1}^d q_i (\partial_i\nabla^2K(t))\\
			&=\left[\rho_{\sigma}''(|t|)-\frac{\rho_{\sigma}'(|t|)}{|t|}\right]\frac{1}{|t|^2}\left(tq^T+qt^T+\langle q,t\rangle\left(Id-3\frac{tt^T}{|t|^2}\right)\right)+\rho'''_{\sigma}(|t|)\frac{\langle q,t\rangle}{|t|}\frac{tt^T}{|t|^2}\,.
			\end{aligned}
		\end{equation*} 
		Hence  
		\begin{equation*}
			\begin{aligned}
				\left\|K^{(12)}(x,x')\right\|_{x,x'}&=\underset{|q|_x\leq 1,~|v_i|_x\leq 1}{\mathrm{sup}}~\left|v_1^T\left(\sum\limits_{i=1}^d q_i (\partial_i\nabla^2K(t))\right)v_2\right|\\
				&=\left(\frac{\sigma}{\sqrt{-\rho''(0)}}\right)^3 \underset{|q|\leq 1,~|v_i|\leq 1}{\mathrm{sup}}~\left|v_1^T\left(\sum\limits_{i=1}^d q_i (\partial_i\nabla^2K(t))\right)v_2\right|\\
				&\leq \left(\frac{\sigma}{\sqrt{-\rho''(0)}}\right)^3\left(6\left\|\frac{1}{|\cdot|}\left(\rho''_{\sigma}(|\cdot|)-\frac{\rho'_{\sigma}(|\cdot|)}{|\cdot|}\right)\right\|_{\infty} +\|\rho'''_{\sigma}\|_{\infty}\right)\,.
			\end{aligned}
		\end{equation*}
		Now, we notice that
		\begin{equation*}
			\begin{aligned}
				\left|\frac{1}{|t|}\left(\rho''_{\sigma}(|t|)-\frac{\rho'_{\sigma}(|t|)}{|t|}\right)\right|
				&=\sigma^{-3}\left|\frac{1}{\sigma^{-1}|t|}\left(\rho''(\sigma^{-1}|t|)-\frac{\rho'(\sigma^{-1}|t|)}{\sigma^{-1}|t|}\right)\right|\\
				&\leq \sigma^{-3}~\underset{s> 0}{\mathrm{sup}}\left|\frac{1}{s}\left(\rho''(s)-\frac{\rho'(s)}{s}\right)\right|
			\end{aligned}
		\end{equation*}
		and $\|\rho'''_{\sigma}\|_{\infty}=\sigma^{-3}\|\rho'''\|_{\infty}$, which yields
		\begin{equation*}
			\left\|K^{(12)}(x,x')\right\|_{x,x'}\leq (-\rho''(0))^{-3/2}\left(6~\underset{s> 0}{\mathrm{sup}}\left|\frac{1}{s}\left(\rho''(s)-\frac{\rho'(s)}{s}\right)\right|+\|\rho'''\|_{\infty}\right).
		\end{equation*}
		
		By \cref{rho_prime,rho_second} we get 
		\begin{equation}\label{eq:kpp}
			\begin{aligned}
				\rho''(s)-\rho'(s)/s&=(2/s)^{d/2}\Gamma(d/2+1)J_{d/2+2}(s).
			\end{aligned}
		\end{equation} 
		Moreover, using \cref{rho_second,bessel_derivative}, we obtain
		\begin{equation}
			\begin{aligned}
				\rho'''(s)&=(2/s)^{d/2}\Gamma(d/2+1)\bigg(-\frac{d}{2}\frac{J_{d/2+2}(s)}{s}+\frac{d}{2}\frac{J_{d/2+1}(s)}{s^2}+J_{d/2+2}'(s)-\frac{J_{d/2+1}'(s)}{s}+\frac{J_{d/2+1}(s)}{s^2}\bigg)\\
				&=(2/s)^{d/2}\Gamma(d/2+1)\bigg(-\frac{d}{2}\frac{J_{d/2+2}(s)}{s}+\Big(\frac{d}{2}+1\Big)\frac{J_{d/2+1}(s)}{s^2}+\frac{d/2+2}{s}J_{d/2+2}(s)-J_{d/2+3}(s)\\
				&\qquad\qquad\qquad\qquad\qquad~\,-\Big(\frac{d}{2}+1\Big)\frac{J_{d/2+1}(s)}{s^2}+\frac{J_{d/2+2}(s)}{s}\bigg)\\
				&=(2/s)^{d/2}\Gamma(d/2+1)\left[-J_{d/2+3}(s)+(3/s)J_{d/2+2}(s)\right].
			\end{aligned}
			\label{rho_third}
		\end{equation}
		Applying \Cref{lemma_bound_bessel} yields $B_{21}=\mathcal{O}(1)$.
		\item $\bm{B_{22}}$: first, we have: \begin{equation*}
	\begin{aligned}
			\|K^{(22)}(x,x')\|_{x,x'}&=\underset{|q_k|_x\leq 1,~|v_l|_x\leq 1}{\mathrm{sup}}\left|v_1^T \left(\sum\limits_{i,j=1}^d q_{1,i}q_{2,j} \partial_{ij}\nabla^2 K(t)\right)v_2\right|\\
			&=\left(\frac{\sigma}{\sqrt{-\rho''(0)}}\right)^4 \underset{|q_k|_2\leq 1}{\mathrm{sup}}\left\|\sum\limits_{i,j=1}^d q_{1,i}q_{2,j}\partial_{ij}\nabla^2 K(t)\right\|_2.
			\end{aligned}
		\end{equation*}
		Then, we recall that $\partial_i \nabla^2K(t)=a_i(t)b_i(t)+c_i(t)$ with
		\begin{equation*}
			\begin{aligned}
			a_i(t)&\eqdef\frac{1}{|t|^2}\left[\rho_{\sigma}''(|t|)-\frac{\rho_{\sigma}'(|t|)}{|t|}\right],\\
			b_i(t)&\eqdef\left(te_i^T+e_it^T+t_i\left(Id-3\frac{tt^T}{|t|^2}\right)\right),\\
			c_i(t)&\eqdef\rho'''_{\sigma}(|t|)\frac{t_i}{|t|}\frac{tt^T}{|t|^2}\,.
			\end{aligned}
		\end{equation*}
		Denoting by $E_{ij}\in\RR^{d\times d}$ the matrix with entry $1$ at row $i$ and column $j$ and zero otherwise, we have
		\begin{equation*}
			\begin{aligned}
				\partial_j a_i(t)&=\frac{1}{|t|^2}\left[\rho_{\sigma}'''(|t|)\frac{t_j}{|t|}-\frac{t_j}{|t|^2}\left(\rho''_{\sigma}(|t|)-\frac{\rho'_{\sigma}(|t|)}{|t|}\right)\right]-2\frac{t_j}{|t|^4}\left[\rho_{\sigma}''(|t|)-\frac{\rho_{\sigma}'(|t|)}{|t|}\right]\\
				&=\frac{t_j}{|t|^3}\left[\rho_{\sigma}'''(|t|)-\frac{1}{|t|}\left(\rho''_{\sigma}(|t|)-\frac{\rho'_{\sigma}(|t|)}{|t|}\right)-\frac{2}{|t|}\left(\rho_{\sigma}''(|t|)-\frac{\rho_{\sigma}'(|t|)}{|t|}\right)\right]\\
				&=\frac{t_j}{|t|^3}\left[\rho_{\sigma}'''(|t|)-\frac{3}{|t|}\left(\rho''_{\sigma}(|t|)-\frac{\rho'_{\sigma}(|t|)}{|t|}\right)\right],\\
			\partial_j b_i(t)&=E_{ji}+E_{ij}+\delta_{ij}\left(Id-3\frac{tt^T}{|t|^2}\right)-3\frac{t_i}{|t|^2}\left(t e_j^T+e_jt^T-2t_j\frac{tt^T}{|t|^2}\right),\\
			\partial_jc_i(t)&=~~~\,\rho^{(4)}_{\sigma}(|t|)\frac{t_i t_j}{|t|^2}\frac{tt^T}{|t|^2}\\
		&~~~\,+\rho'''_{\sigma}(|t|)\left[\left(\delta_{ij}-\frac{t_i t_j}{|t|^2}\right)\frac{1}{|t|}\frac{tt^T}{|t|^2}+\frac{t_i}{|t|}\frac{1}{|t|^2}\left(te_j^T+e_jt^T+t_j\left(Id-3\frac{tt^T}{|t|^2}\right)\right)\right],
			\end{aligned}
		\end{equation*}
	and hence
		\begin{equation*}
			\begin{aligned}
				&\sum\limits_{i,j=1}^d q_{1,i}q_{2,j}(\partial_j a_i(t))b_i(t)=\frac{\langle q_2,t\rangle}{|t|^3}\left[\rho_{\sigma}'''(|t|)-\frac{3}{|t|}\left(\rho''_{\sigma}(|t|)-\frac{\rho'_{\sigma}(|t|)}{|t|}\right)\right]\left(tq_1^T+q_1t^T+\langle q_1,t\rangle\left(Id-3\frac{tt^T}{|t|^2}\right)\right),\\
				&\sum\limits_{i,j=1}^d q_{1,i}q_{2,j}a_i(t)(\partial_j b_i(t))=a_i(t)\left(q_1q_2^T+q_2q_1^T+\langle q_1,q_2\rangle\left(Id-3\frac{tt^T}{|t|^2}\right)-3\frac{\langle q_1,t\rangle}{|t|^2}\left(tq_2^T+q_2t^T-2\langle q_2,t\rangle \frac{tt^T}{|t|^2}\right)\right),\\
				&\sum\limits_{i,j=1}^d q_{1,i} q_{2,j}\partial_jc_i(t)=\rho^{(4)}_{\sigma}(|t|)\frac{\langle q_1,t\rangle \langle q_2,t\rangle}{|t|^2}\frac{tt^T}{|t|^2}\\
				&+\rho'''_{\sigma}(|t|)\left[\left(\langle q_1,q_2\rangle-\frac{\langle q_1,t\rangle \langle q_2,t\rangle}{|t|^2}\right)\frac{1}{|t|}\frac{tt^T}{|t|^2}+\frac{\langle q_1,t\rangle}{|t|}\frac{1}{|t|^2}\left(t q_2^T+q_2t^T+\langle q_2,t\rangle\left(Id-3\frac{tt^T}{|t|^2}\right)\right)\right].
			\end{aligned}
		\end{equation*}
 		This ultimately yields
		\begin{equation*}
			\begin{aligned}
				\underset{|q_k|_2\leq 1}{\mathrm{sup}}\left\|\sum\limits_{i,j=1}^d q_{1,i} q_{2,j}(\partial_j a_i(t))b_i(t)\right\|_2&\leq 6\left\|\frac{1}{|\cdot|}\left(\rho'''_{\sigma}(|\cdot|)-\frac{3}{|\cdot|}\left(\rho''_{\sigma}(|\cdot|)-\frac{\rho'_{\sigma}(|\cdot|)}{|\cdot|}\right)\right)\right\|_{\infty}\\
				&\leq 6\,\sigma^{-4}\left(\underset{s> 0}{\mathrm{sup}}\left|\frac{\rho'''(s)}{s}\right|+3~\underset{s> 0}{\mathrm{sup}}\left|\frac{1}{s^2}\left(\rho''(s)-\frac{\rho'(s)}{s}\right)\right|\right),\\
				\underset{|q_k|_2\leq 1}{\mathrm{sup}}\left\|\sum\limits_{i,j=1}^d q_{1,i} q_{2,j}a_i(t)\partial_jb_i(t)\right\|_2&\leq 18\left\|\frac{1}{|\cdot|^2}\left(\rho''_{\sigma}(|\cdot|)-\frac{\rho'_{\sigma}(|\cdot|)}{|\cdot|}\right)\right\|_{\infty}\\
				&=18\,\sigma^{-4}\,\underset{s>0}{\mathrm{sup}}\left|\frac{1}{s^2}\left(\rho''(s)-\frac{\rho'(s)}{s}\right)\right|,\\
				\underset{|q_k|_2\leq 1}{\mathrm{sup}}\left\|\sum\limits_{i,j=1}^d q_{1,i} q_{2,j}\partial_jc_i(t)\right\|_2&\leq \|\rho^{(4)}_{\sigma}\|_{\infty}+8\left\|\frac{\rho'''_{\sigma}(|\cdot|)}{|\cdot|}\right\|_{\infty}\\
				&=\sigma^{-4}\left(\|\rho^{(4)}\|_{\infty}+\underset{s> 0}{\mathrm{sup}}\left|\frac{\rho'''(s)}{s}\right|\right),\\
			\end{aligned}
		\end{equation*}
		and we finally obtain
		\begin{equation*}
			B_{22}\leq (-\rho''(0))^{-2}\left(\|\rho^{(4)}\|_{\infty}+7~\underset{s> 0}{\mathrm{sup}}\left|\frac{\rho'''(s)}{s}\right|+36~\underset{s> 0}{\mathrm{sup}}\left|\frac{1}{s^2}\left(\rho''(s)-\frac{\rho'(s)}{s}\right)\right| \right).
		\end{equation*}
		In view of \cref{eq:kpp,rho_third}, another application of \Cref{lemma_bound_bessel} therefore yields $B_{22}=\mathcal{O}(1)$. \qedhere
	\end{itemize}
\end{proof}

\subsection{Near region bound}
\begin{lemma}
	For all $\mathfrak{d}_{\mathfrak{g}}(x,x')\leq r_{\mathrm{near}}\eqdef1/\sqrt{5}$ and all $v\in\mathbb{C}^d$ we have $-K^{(02)}(x,x')[v,v]\geq \overline{\epsilon}_2\,|v|_x^2$ where~${\overline{\epsilon}_2\eqdef 0.6<r_{\rm near}^{-2}}$.
\end{lemma}
\begin{proof}
	Let $r>0$ and $x,x'$ be such that $\mathfrak{d}_{\mathfrak{g}}(x,x')\leq r$. For any $v\in\mathbb{C}^d$, defining $t=x-x'$, we have
	\begin{equation*}
		\begin{aligned}
			-K^{(02)}(x,x')[v,v]&=-v^T \nabla^2 K(t)v\\
			&=-\frac{\rho'(\sigma^{-1}|t|)}{\sigma^{-1}|t|}\sigma^{-2}|v|^2-\left[\sigma^{-2}\rho''(\sigma^{-1}|t|)-\sigma^{-1}\frac{\rho'(\sigma^{-1}|t|)}{|t|}\right]\frac{(v\cdot t)^2}{|t|^2}\\
			&\geq\left(-\frac{\rho'(\sigma^{-1}|t|)}{\sigma^{-1}|t|}-\left|\rho''(\sigma^{-1}|t|)-\frac{\rho'(\sigma^{-1}|t|)}{\sigma^{-1}|t|}\right|\right)\sigma^{-2}|v|^2\\
			&\geq \overline{\epsilon}_2(r)\,|v|_x^2,
		\end{aligned}
	\end{equation*}
	where 
 \begin{equation*}
     \begin{aligned}
     \overline{\epsilon}_2(r)\eqdef&\frac{1}{(-\rho''(0))}\,\mathrm{inf}\left\{-\frac{\rho'(s)}{s}-\left|\rho''(s)-\frac{\rho'(s)}{s}\right|,~s\sqrt{-\rho''(0)}\in (-r,r)\right\}\\
     =&\frac{1}{(-\rho''(0))}\,\mathrm{inf}\left\{-\rho''(0)+\rho''(0)-\frac{\rho'(s)}{s}-\left|\rho''(s)-\frac{\rho'(s)}{s}\right|,~s\sqrt{-\rho''(0)}\in (-r,r)\right\}.
     \end{aligned}
\end{equation*}

	Proceeding as in the proof of \Cref{lemma_bound_bessel}, we obtain:
	\begin{equation*}
		\left|\frac{\rho'(s)}{s}-\rho''(0)\right|\leq (-\rho''(0))\frac{2(e^{s^2/4}-1)}{d+4}~~\mathrm{and}~~\left|\rho''(s)-\frac{\rho'(s)}{s}\right|\leq (-\rho''(0))\frac{s^2 e^{s^2/4}}{d+4}\,.
	\end{equation*}
	If $s\leq 1$ then 
	\begin{equation*}
		\frac{2(e^{s^2/4}-1)}{d+4}\leq \frac{e^{1/4}-1}{3}\leq 0.1~~\mathrm{and}~~\frac{s^2e^{s^2/4}}{d+4}\leq \frac{e^{1/4}}{6}\leq 0.3\,,
	\end{equation*}
	which shows $\overline{\epsilon}_2(r)\geq 0.6$ as soon as $r\leq 1/\sqrt{5}\leq 1/\sqrt{d+2}$.
\end{proof}

\subsection{Far region bound}

\begin{lemma}
	For all $\mathfrak{d}_{\mathfrak{g}}(x,x')\geq r_{near}$ we have $|K(x,x')|\leq 1 - \overline{\epsilon}_0$ where $\overline{\epsilon}_0\eqdef 0.07$.
\end{lemma}
\begin{proof}
	If $\mathfrak{d}_{\mathfrak{g}}(x,x')=\sigma^{-1}|x-x'|/\sqrt{d+2}\geq r_{near}=1/\sqrt{5}$ then $\sigma^{-1}|x-x'|\geq r_{near}\sqrt{d+2}\geq 2/\sqrt{5}$. We hence need to prove that $\rho(s)\leq 1-\overline{\epsilon_0}$ for every $s\geq 2/\sqrt{5}$. Using \Cref{lemma_bound_bessel}, we obtain, for every $s>0$:
	\begin{equation*}
		\rho(s)\leq 0.8\frac{2^{d/2}\Gamma(d/2+1)}{s^{1/3+d/2}}\leq 0.8 \frac{2^{3/2}\Gamma(3/2+1)}{s^{4/3}}\,.
	\end{equation*}
	The right hand side is smaller than $0.9$ on $(5/2,+\infty)$, but is not always smaller than $1$ on $[2/\sqrt{5},5/2]$. On~$[2/\sqrt{5},5/2]$, we use the other bound provided by \Cref{lemma_bound_bessel}, namely:
	\begin{equation*}
		\rho(s)\leq b(s)\eqdef 1-\frac{1}{d/2+1}\frac{s^2}{4}+\frac{1}{(d/2+1)(d/2+2)}\left(e^{s^2/4}-1-\frac{s^2}{4}\right)\,.
	\end{equation*}
	We notice that $b$ is decreasing on $[2/\sqrt{5},2\sqrt{\mathrm{log}(3+d/2)}]$ and increasing on $[2\sqrt{\mathrm{log}(3+d/2)},5/2]$. Moreover, we have $b(2/\sqrt{5})\leq 0.93$ and ${b(5/2)\leq 0.63}$. Hence, we finally obtain $\rho(s)\leq 0.93$ for every $s\geq 2/\sqrt{5}$.
\end{proof}

We now bound the kernel width $W(h,s)$ of $K$.
\begin{lemma}
	If $h=\mathcal{O}(1)$ then $W(h,s)=\mathcal{O}(s^{2/(d+1)})$.
\end{lemma}
\begin{proof}
	From the computations conducted in the proof of \Cref{lemma_uniform_bounds} and the fact $J_{d/2}(t)=\mathcal{O}(t^{-1/2})$ as~$t\to+\infty$, we obtain that, for every $i,j\in\{0,1,2\}$ with $i+j\leq 3$,
	\begin{equation*}
		\left\|K^{(ij)}(x,x')\right\|_{x,x'}\leq \frac{1}{(\sigma^{-1}t)^{(d+1)/2}}
	\end{equation*}
	as soon as $\sigma^{-1}\|x-x'\|\gtrsim t$. As a result, provided $\Delta\gtrsim s^{2/(d+1)}$, for every $(x_k)_{k=1}^s\in\mathcal{S}_{\Delta}$ we have
\begin{equation*}
\sum\limits_{k=2}^s \left\|K^{(ij)}(x_1,x_k)\right\|_{x_1,x_k}\lesssim \sum\limits_{k=2}^s \frac{1}{s}=\mathcal{O}(1)\,.
\end{equation*}
\end{proof}

\subsection{Gradient bounds}
For $j\in\{0,1,2\}$, we have $L_j(\omega)=\|D_j[\varphi_{\omega}](x)\|_x=\big(\sigma\|\omega\|_2/\sqrt{-\rho''(0)}\big)^j$. Moreover, using the notations of~\cite[Section 5.1]{poonGeometryOfftheGridCompressed2023}, we have
\begin{equation*}
	\begin{aligned}
	\|D_2[\varphi_{\omega}](x)-D_2[\varphi_{\omega}](x')[\tau_{x\to x'}\cdot,\tau_{x\to x'}\cdot]\|_x&=\frac{\sigma^{2}}{(-\rho''(0))}\|\omega\|_{2}^2\,|e^{-i\langle \omega,x\rangle}-e^{-i\langle \omega,x'\rangle}|\\
	&\leq \frac{\sigma^{2}}{(-\rho''(0))}\|\omega\|_{2}^2\,|\langle \omega,x-x'\rangle|\\
	&\leq \frac{\sigma^{2}}{(-\rho''(0))}\|\omega\|_{2}^3\,|x-x'|\\
	&=\big(\sigma\|\omega\|_2/\sqrt{-\rho''(0)}\big)^3\,\mathfrak{d}_{\mathfrak{g}}(x,x')\,.
	\end{aligned}
\end{equation*}
This yields $L_3(\omega)\leq \big(\sigma\|\omega\|_2/\sqrt{-\rho''(0)}\big)^3$. Since $\omega\sim\mathcal{U}(B(0,\sigma^{-1}))$, we have $\sigma\|\omega\|_2\leq 1$, and hence $$\|L_j\|_{\infty}\leq (-\rho''(0))^{-j/2}$$
for every $j\in\{0,...,3\}$. We can therefore set $\overline{L}_j=(-\rho''(0))^{-j/2}=\mathcal{O}(1)$ (see point (i) in the remark below Assumption 2 in \cite[Section 5.2]{poonGeometryOfftheGridCompressed2023}).
\end{document}